\documentclass{amsart}
\usepackage{amsmath}
\usepackage{amssymb}
\usepackage{amsfonts}
\usepackage{graphicx}
\usepackage{tikz}
\usepackage{longtable}
\usepackage{array}

\usepackage{lipsum} 
\usepackage[T1]{fontenc}
\newtheorem{theorem}{Theorem}
\newtheorem{lemma}{Lemma}

\newtheorem{example}{Example}
\newtheorem{remark}{Remark}

\tikzset{
    main_line/.style = {
        draw=black,
        line width=1pt
    },
    select_line/.style = {
        draw=gray,
        line width=1pt
    },
    select_back_line/.style = {
        draw=gray,
        line width=1pt,
        dashed,
        line cap=round
    },
    back_line/.style = {
        draw=black,
        line width=1pt,
        dashed,
        line cap=round
    },
    point/.style = {
        draw=black,
        line width=0pt,
        fill
    },
    select_point/.style = {
        draw=gray,
        line width=0pt,
        fill=gray
    },
}

\title{Quadrangulation moves}
\author{Korablev Ph. G.}

\address{Chelyabinsk State University, Chelyabinsk, Russia; N.N. Krasovsky Institute of Mathematics and Meckhanics, Ekaterinburg, Russia}
\email{korablev@csu.ru}
\date{}

\begin{document}

\begin{abstract}
    We prove that any two quadrangulations of the same surface, i.e., decompositions of a surface into squares, can be transformed into one another by a finite sequence of two-dimensional local cubical Pachner moves and non-local subdivision transformations.
\end{abstract}

\maketitle

\section{Introduction}

It is well known that any two triangulations of any $n$-manifold can be transformed into one another by a finite sequence of local Pachner moves (\cite{FH,L,P,RST}). Each move replaces a subset of $n$-faces of an $(n+1)$-dimensional simplex with the complementary subset of faces.

Analogously to triangulations, a cubulation is a presentation of a manifold as a result of gluing some $n$-cubes along their $(n-1)$-faces. The classification of 3-manifolds with low cubical complexity was carried out in \cite{A1,A2,KK}. Sometimes, cubulations are called cubications.

Similar to classical Pachner moves of triangulations, we can consider cubical Pachner moves for cubulations. Each such move replaces a set of $n$-faces of an $(n+1)$-cube by the complementary set of faces. It is clear that these moves cannot transform one cubulation of an $n$-manifold into another, because the parity of cubes is conserved under cubical Pachner moves. Some characterisations of cubulation classes modulo cubical Pachner moves were studied in \cite{F1,F2}.

In the 2-dimensional case, when the manifolds are surfaces, cubulations are called quadrangulations. In this case, a complete set of moves has been found in \cite{N}. It consists of two types of diagonal switches. These switches are not related to cubical Pachner moves. The results of \cite{N} apply not to all surfaces, but only to closed ones, and furthermore only to so-called bipartite quadrangulations.

The main subject of this paper is to prove that any two quadrangulations of the same surface can be transformed into each other by a finite sequence of 2-dimensional cubical Pachner moves and one additional non-local transformation --- subdivision.

The structure of the paper is as follows. In Section 2, we give formal definitions for cubulations and describe how to construct a cubulation of a manifold from any triangulation. We also describe cubical Pachner moves and the subdivision transformation. For the notation of moves, we follow the approach of \cite{BEE}. All definitions and constructions apply to the general $n$-dimensional case, although the main result applies to the case $n = 2$.

In Section 3, we prove two intermediate lemmas and the main theorem.

\section{Cubulations}

Let $I = [-1, 1]$ be the interval. By an \emph{$n$-dimensional cube} (or simply \emph{$n$-cube}) we mean $C^n = I\times\ldots\times I$, where the product is taken exactly $n$ times. A 1-cube is simply an interval, a 2-cube is a square, a 3-cube is a standard three-dimensional cube, and so on.

Each $n$-cube $C^n$ has $2n$ $(n - 1)$-dimensional faces: $\partial_i^{-} C^n = I\times \ldots \{-1\}\times\ldots I$, $i = 1, \ldots, n$, and $\partial_i^{+} C^n = I\times \ldots \{1\}\times\ldots I$, $i = 1, \ldots, n$ (see fig. \ref{Fig:Faces}). It is clear that each $(n - 1)$-dimensional face is an $(n - 1)$-cube. Each of these in turn has $(n - 2)$-dimensional faces, and so on.

\begin{figure}[ht]
    \begin{center}
        \begin{tikzpicture}[scale=1.7, baseline=(current bounding box.center)]
            \path[main_line] (0.411, 28.464) -- (1.306, 28.464);
            \path[point] (0.43, 28.464) node[left] {$\partial_1^{-} C^1$} arc(0.0:90.0:0.025 and -0.025)arc(90.0:180.0:0.025 and -0.025)arc(180.0:270.0:0.025 and -0.025)arc(270.0:360.0:0.025 and -0.025) -- cycle;
            \path[point] (1.344, 28.464) node[right] {$\partial_1^{+} C^1$}arc(0.0:90.0:0.025 and -0.025)arc(90.0:180.0:0.025 and -0.025)arc(180.0:270.0:0.025 and -0.025)arc(270.0:360.0:0.025 and -0.025) -- cycle;
        \end{tikzpicture}
        \hfill
        \begin{tikzpicture}[scale=1.7, baseline=(current bounding box.center)]
            \path[main_line] (1.6, 28.313) -- (2.495, 28.313) node[midway, below] {$\partial_2^{-}C^2$};
            \path[main_line] (1.6, 29.208) -- (2.495, 29.208) node[midway, above] {$\partial_2^{+}C^2$};
            \path[main_line] (2.495, 28.313) -- (2.495, 29.208)  node[midway, right] {$\partial_1^{+}C^2$};
            \path[main_line] (1.6, 29.208) -- (1.6, 28.313)  node[midway, left] {$\partial_1^{-}C^2$};
            
            \path[point] (1.625, 28.313)arc(0.0:90.0:0.025 and -0.025)arc(90.0:180.0:0.025 and -0.025)arc(180.0:270.0:0.025 and -0.025)arc(270.0:360.0:0.025 and -0.025) -- cycle;
            \path[point] (1.625, 29.208)arc(0.0:90.0:0.025 and -0.025)arc(90.0:180.0:0.025 and -0.025)arc(180.0:270.0:0.025 and -0.025)arc(270.0:360.0:0.025 and -0.025) -- cycle;
            \path[point] (2.52, 29.208)arc(0.0:90.0:0.025 and -0.025)arc(90.0:180.0:0.025 and -0.025)arc(180.0:270.0:0.025 and -0.025)arc(270.0:360.0:0.025 and -0.025) -- cycle;
            \path[point] (2.52, 28.313)arc(0.0:90.0:0.025 and -0.025)arc(90.0:180.0:0.025 and -0.025)arc(180.0:270.0:0.025 and -0.025)arc(270.0:360.0:0.025 and -0.025) -- cycle;
        \end{tikzpicture}
        \hfill
        \begin{tikzpicture}[scale=1.9, baseline={([yshift=-7pt]current bounding box.center)}]
            \path[point] (3.018, 28.048)arc(0.0:90.0:0.025 and -0.025)arc(90.0:180.0:0.025 and -0.025)arc(180.0:270.0:0.025 and -0.025)arc(270.0:360.0:0.025 and -0.025) -- cycle;
            \path[point] (3.018, 28.943)arc(0.0:90.0:0.025 and -0.025)arc(90.0:180.0:0.025 and -0.025)arc(180.0:270.0:0.025 and -0.025)arc(270.0:360.0:0.025 and -0.025) -- cycle;
            \path[point] (3.458, 29.197)arc(0.0:90.0:0.025 and -0.025)arc(90.0:180.0:0.025 and -0.025)arc(180.0:270.0:0.025 and -0.025)arc(270.0:360.0:0.025 and -0.025) -- cycle;
            \path[point] (3.91, 28.948)arc(0.0:90.0:0.025 and -0.025)arc(90.0:180.0:0.025 and -0.025)arc(180.0:270.0:0.025 and -0.025)arc(270.0:360.0:0.025 and -0.025) -- cycle;
            \path[point] (4.353, 29.197)arc(0.0:90.0:0.025 and -0.025)arc(90.0:180.0:0.025 and -0.025)arc(180.0:270.0:0.025 and -0.025)arc(270.0:360.0:0.025 and -0.025) -- cycle;
            \path[point] (3.914, 28.048)arc(0.0:90.0:0.025 and -0.025)arc(90.0:180.0:0.025 and -0.025)arc(180.0:270.0:0.025 and -0.025)arc(270.0:360.0:0.025 and -0.025) -- cycle;
            \path[point] (4.353, 28.302)arc(0.0:90.0:0.025 and -0.025)arc(90.0:180.0:0.025 and -0.025)arc(180.0:270.0:0.025 and -0.025)arc(270.0:360.0:0.025 and -0.025) -- cycle;
            \path[point] (3.458, 28.302)arc(0.0:90.0:0.025 and -0.025)arc(90.0:180.0:0.025 and -0.025)arc(180.0:270.0:0.025 and -0.025)arc(270.0:360.0:0.025 and -0.025) -- cycle;
            
            \path[main_line] (2.993, 28.048) -- (3.888, 28.048) node[midway, below] {$\partial_3^{-}C^3$};
            \path[main_line] (2.993, 28.943) -- (3.888, 28.943);
            \path[main_line] (3.888, 28.048) -- (3.888, 28.943);
            \path[main_line] (2.993, 28.943) -- (2.993, 28.048) node[midway, left] {$\partial_1^{-}C^3$};
            \path[main_line] (4.327, 28.302) -- (3.888, 28.048);
            \path[main_line] (4.327, 29.197) -- (3.888, 28.943);
            \path[back_line] (3.432, 28.302) -- (2.993, 28.048);
            \path[main_line] (3.432, 29.197) -- (2.993, 28.943);
            \path[back_line] (3.432, 28.302) -- (4.327, 28.302);
            \path[main_line] (3.432, 29.197) -- (4.327, 29.197) node[midway, above] {$\partial_3^{+}C^3$};
            \path[main_line] (4.327, 28.302) -- (4.327, 29.197)  node[midway, right] {$\partial_1^{+}C^3$};
            \path[back_line] (3.432, 29.197) -- (3.432, 28.302);

            \draw (3.66, 28.6) node {{\small $\partial_2^{-}C^3$}};
            \draw (4.1, 28.8) node {{\small $\partial_2^{+}C^3$}};
        \end{tikzpicture}
    \end{center}
    \caption{\label{Fig:Faces}$n$-cubes ($n = 1$ on the left, $n = 2$ in the center, $n = 3$ on the right) and their $(n - 1)$-faces}
\end{figure}
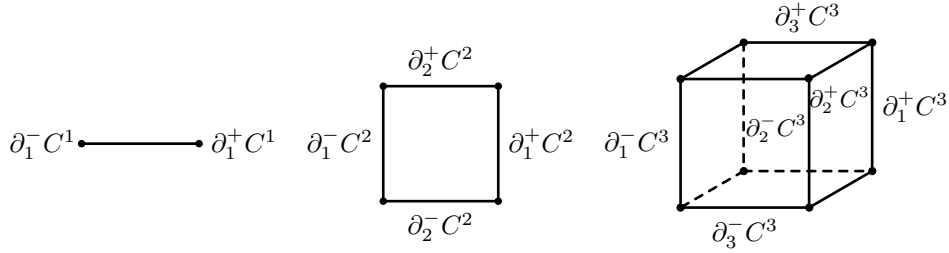

By \emph{cubulation} of an $n$-manifold $M$ we mean the presentation of $M$ as the result of gluing $n$-cubes via homeomorphisms between their $(n - 1)$-dimensional faces that map vertices to vertices, edges to edges, and so on. We call two cubulations $\mathcal{C}_1$ and $\mathcal{C}_2$ of the manifold $M$ \emph{equivalent} if there exists a homeomorphism $\varphi\colon M\to M$ which preserves the faces of the cubulation, i.e., maps each $k$-face of $\mathcal{C}_1$ to a $k$-face of $\mathcal{C}_2$.

It is easy to show that if an $n$-manifold $M$ admits a triangulation, then it admits a cubulation. Indeed, consider the following procedure, which transforms any triangulation $\mathcal{T}$ of the $n$-manifold $M$ into a cubulation $\mathcal{C}$ of $M$. We define the procedure inductively. Each 1-simplex of $\mathcal{T}$ (an edge) is split into two 1-cubes by the center of this simplex (see fig. \ref{Fig:CubulationFromTriangulation} on the left). This center is the 1-dimensional splitting set. Each 2-simplex of $\mathcal{T}$ (a triangle) is split into three 2-cubes by taking the cone with the center in the triangle's center over the union of all 1-dimensional splitting sets (see fig. \ref{Fig:CubulationFromTriangulation} in the center). This cone is the 2-dimensional splitting set. Next, each 3-simplex of $\mathcal{T}$ (a tetrahedron) is split into four 3-cubes by taking a cone with the center in the tetrahedron's center over the union of all 2-dimensional splitting sets. This cone is the 3-dimensional splitting set (see fig. \ref{Fig:CubulationFromTriangulation} on the right). In general: for any dimension $k\in\{1, \ldots, n\}$ we split each $k$-simplex of $\mathcal{T}$ into $k + 1$ $k$-cubes by taking the cone with the center in the center of the simplex over the union of all $(k - 1)$-dimensional splitting sets built at the previous step. It is clear that if the triangulation $\mathcal{T}$ has $x$ $n$-simplices, then the constructed cubulation $\mathcal{C}$ contains $(n + 1)\cdot x$ $n$-cubes. Denote $\mathcal{C} = \beta(\mathcal{T})$. Thus, the map $\beta$ transforms each triangulation of the manifold into a corresponding cubulation of the same manifold.

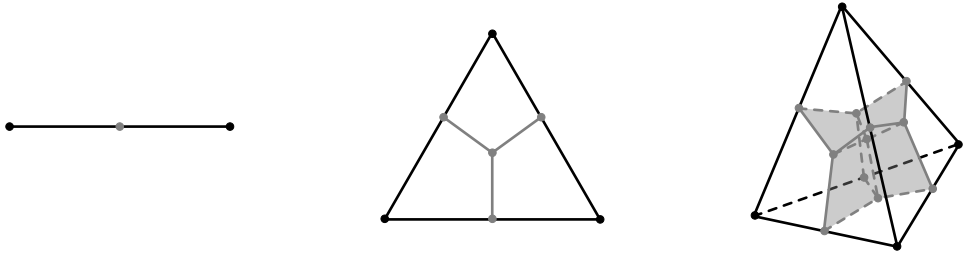
\begin{figure}[ht]
    \begin{center}
        \begin{tikzpicture}[scale=2.0, baseline=(current bounding box.center)]
            \path[main_line] (0.425, 26.086) -- (1.858, 26.086);
            \path[point] (0.438, 26.086)arc(0.0:90.0:0.025 and -0.025)arc(90.0:180.0:0.025 and -0.025)arc(180.0:270.0:0.025 and -0.025)arc(270.0:360.0:0.025 and -0.025) -- cycle;
            \path[select_point] (1.167, 26.086)arc(0.0:90.0:0.025 and -0.025)arc(90.0:180.0:0.025 and -0.025)arc(180.0:270.0:0.025 and -0.025)arc(270.0:360.0:0.025 and -0.025) -- cycle;
            \path[point] (1.896, 26.086)arc(0.0:90.0:0.025 and -0.025)arc(90.0:180.0:0.025 and -0.025)arc(180.0:270.0:0.025 and -0.025)arc(270.0:360.0:0.025 and -0.025) -- cycle;
        \end{tikzpicture}
        \hfill
        \begin{tikzpicture}[scale=2.0, baseline=(current bounding box.center)]
            \path[point] (2.405, 25.75)arc(0.0:90.0:0.025 and -0.025)arc(90.0:180.0:0.025 and -0.025)arc(180.0:270.0:0.025 and -0.025)arc(270.0:360.0:0.025 and -0.025) -- cycle;
            \path[point] (3.83, 25.746)arc(0.0:90.0:0.025 and -0.025)arc(90.0:180.0:0.025 and -0.025)arc(180.0:270.0:0.025 and -0.025)arc(270.0:360.0:0.025 and -0.025) -- cycle;
            \path[point] (3.117, 26.974)arc(0.0:90.0:0.025 and -0.025)arc(90.0:180.0:0.025 and -0.025)arc(180.0:270.0:0.025 and -0.025)arc(270.0:360.0:0.025 and -0.025) -- cycle;
            
            \path[main_line] (3.804, 25.746) -- (2.379, 25.746) -- (3.091, 26.981) -- cycle;
            \path[select_line] (2.769, 26.422) -- (3.091, 26.187);
            \path[select_line] (3.414, 26.422) -- (3.091, 26.187);
            \path[select_line] (3.091, 25.746) -- (3.091, 26.187);

            \path[select_point] (3.117, 25.75)arc(0.0:90.0:0.025 and -0.025)arc(90.0:180.0:0.025 and -0.025)arc(180.0:270.0:0.025 and -0.025)arc(270.0:360.0:0.025 and -0.025) -- cycle;
            \path[select_point] (3.44, 26.422)arc(0.0:90.0:0.025 and -0.025)arc(90.0:180.0:0.025 and -0.025)arc(180.0:270.0:0.025 and -0.025)arc(270.0:360.0:0.025 and -0.025) -- cycle;
            \path[select_point] (2.794, 26.422)arc(0.0:90.0:0.025 and -0.025)arc(90.0:180.0:0.025 and -0.025)arc(180.0:270.0:0.025 and -0.025)arc(270.0:360.0:0.025 and -0.025) -- cycle;
            \path[select_point] (3.117, 26.187)arc(0.0:90.0:0.025 and -0.025)arc(90.0:180.0:0.025 and -0.025)arc(180.0:270.0:0.025 and -0.025)arc(270.0:360.0:0.025 and -0.025) -- cycle;
        \end{tikzpicture}
        \hfill
        \begin{tikzpicture}[scale=2.0, baseline=(current bounding box.center)]
            \path[back_line] (4.31, 25.716) -- (5.671, 26.195);

            \path[fill=gray, fill opacity = 0.4] (4.609, 26.431) -- (4.837, 26.125) -- (4.774, 25.617) -- (5.132, 25.835) -- (5.493, 25.897) -- (5.305, 26.335) -- (5.323, 26.609) -- (4.986, 26.396) -- cycle;

            \path[point] (4.343, 25.721)arc(0.0:90.0:0.025 and -0.025)arc(90.0:180.0:0.025 and -0.025)arc(180.0:270.0:0.025 and -0.025)arc(270.0:360.0:0.025 and -0.025) -- cycle;
            \path[point] (5.283, 25.516)arc(0.0:90.0:0.025 and -0.025)arc(90.0:180.0:0.025 and -0.025)arc(180.0:270.0:0.025 and -0.025)arc(270.0:360.0:0.025 and -0.025) -- cycle;
            \path[point] (5.689, 26.19)arc(0.0:90.0:0.025 and -0.025)arc(90.0:180.0:0.025 and -0.025)arc(180.0:270.0:0.025 and -0.025)arc(270.0:360.0:0.025 and -0.025) -- cycle;
            \path[point] (4.92, 27.101)arc(0.0:90.0:0.025 and -0.025)arc(90.0:180.0:0.025 and -0.025)arc(180.0:270.0:0.025 and -0.025)arc(270.0:360.0:0.025 and -0.025) -- cycle;

            \path[select_back_line] (4.605, 26.429) -- (4.986, 26.396) -- (5.325, 26.609);
            \path[select_back_line] (5.305, 26.338) -- (5.058, 26.224);
            \path[select_back_line] (5.039, 25.972) -- (5.132, 25.835) -- (5.493, 25.897);
            
            \path[main_line] (4.31, 25.716) -- (4.892, 27.11) -- (5.671, 26.195) -- (5.26, 25.514) -- cycle;
            \path[main_line] (4.892, 27.11) -- (5.26, 25.514);
            
            \path[select_line] (4.609, 26.431) -- (4.837, 26.123) -- (5.082, 26.305) -- (5.305, 26.336) -- (5.323, 26.609);
            \path[select_line] (4.774, 25.617) -- (4.834, 26.123);
            \path[select_line] (5.493, 25.897) -- (5.305, 26.335);
            \path[select_back_line] (5.039, 25.972) -- (4.989, 26.393);
            \path[select_back_line] (5.056, 26.227) -- (4.989, 26.396);
            \path[select_back_line] (4.837, 26.125) -- (5.06, 26.228);
            \path[select_back_line] (5.058, 26.224) -- (5.132, 25.835);
            \path[select_back_line] (4.774, 25.617) -- (5.132, 25.835);

            \path[select_point] (5.083, 26.226)arc(0.0:90.0:0.025 and -0.025)arc(90.0:180.0:0.025 and -0.025)arc(180.0:270.0:0.025 and -0.025)arc(270.0:360.0:0.025 and -0.025) -- cycle;
            \path[select_point] (4.634, 26.431)arc(0.0:90.0:0.025 and -0.025)arc(90.0:180.0:0.025 and -0.025)arc(180.0:270.0:0.025 and -0.025)arc(270.0:360.0:0.025 and -0.025) -- cycle;
            \path[select_point] (5.105, 26.302)arc(0.0:90.0:0.025 and -0.025)arc(90.0:180.0:0.025 and -0.025)arc(180.0:270.0:0.025 and -0.025)arc(270.0:360.0:0.025 and -0.025) -- cycle;
            \path[select_point] (5.345, 26.609)arc(0.0:90.0:0.025 and -0.025)arc(90.0:180.0:0.025 and -0.025)arc(180.0:270.0:0.025 and -0.025)arc(270.0:360.0:0.025 and -0.025) -- cycle;
            \path[select_point] (5.065, 25.972)arc(0.0:90.0:0.025 and -0.025)arc(90.0:180.0:0.025 and -0.025)arc(180.0:270.0:0.025 and -0.025)arc(270.0:360.0:0.025 and -0.025) -- cycle;
            \path[select_point] (4.803, 25.618)arc(0.0:90.0:0.025 and -0.025)arc(90.0:180.0:0.025 and -0.025)arc(180.0:270.0:0.025 and -0.025)arc(270.0:360.0:0.025 and -0.025) -- cycle;
            \path[select_point] (5.519, 25.897)arc(0.0:90.0:0.025 and -0.025)arc(90.0:180.0:0.025 and -0.025)arc(180.0:270.0:0.025 and -0.025)arc(270.0:360.0:0.025 and -0.025) -- cycle;
            \path[select_point] (5.157, 25.835)arc(0.0:90.0:0.025 and -0.025)arc(90.0:180.0:0.025 and -0.025)arc(180.0:270.0:0.025 and -0.025)arc(270.0:360.0:0.025 and -0.025) -- cycle;
            \path[select_point] (5.327, 26.337)arc(0.0:90.0:0.025 and -0.025)arc(90.0:180.0:0.025 and -0.025)arc(180.0:270.0:0.025 and -0.025)arc(270.0:360.0:0.025 and -0.025) -- cycle;
            \path[select_point] (4.862, 26.124)arc(0.0:90.0:0.025 and -0.025)arc(90.0:180.0:0.025 and -0.025)arc(180.0:270.0:0.025 and -0.025)arc(270.0:360.0:0.025 and -0.025) -- cycle;
            \path[select_point] (5.014, 26.396)arc(0.0:90.0:0.025 and -0.025)arc(90.0:180.0:0.025 and -0.025)arc(180.0:270.0:0.025 and -0.025)arc(270.0:360.0:0.025 and -0.025) -- cycle;

        \end{tikzpicture}
    \end{center}
    \caption{\label{Fig:CubulationFromTriangulation}Splitting each $n$-simplex into $n + 1$ $n$-cubes ($n = 1$ on the left, $n = 2$ in the center, $n = 3$ on the right), the splitting set is drawn in gray}
\end{figure}

We say that a cubulation $\mathcal{C}$ of $M$ \emph{matches} the triangulation $\mathcal{T}$ of $M$ if $\mathcal{C}$ is equivalent to $\beta(\mathcal{T})$.

It is well known that any two triangulations $\mathcal{T}_1$ and $\mathcal{T}_2$ of an $n$-manifold $M$ can be connected by a sequence of \emph{Pachner} moves. Each such move is a local transformation that replaces a set of $n$-faces of an $(n + 1)$-simplex by the complementary set of faces. It is possible to define similar \emph{cubical Pachner moves} for cubulations. The main idea is that each move replaces a set of $n$-faces of an $(n + 1)$-cube by the complementary set of faces. The details of this definition are as follows.

Let $C^{n + 1}$ be an $(n + 1)$-cube. Its $n$-faces are divided into $n + 1$ pairs of opposite faces: $(\partial_i^{-}C^{n + 1}, \partial_i^{+}C^{n + 1})$, $i = 1, \ldots, n + 1$. Let $S$ be a set of $n$-faces of the cube $C^{n + 1}$. We say that the set $S$ has type $(X, Y)$ if there are exactly $X$ pairs of opposite $n$-faces such that neither face belongs to $S$ (i.e., the pair is disjoint from $S$), and exactly $Y$ pairs such that both faces belong to $S$ (i.e., the pair is fully contained in $S$). It is clear that if $S$ has type $(X, Y)$, then the complementary set of $n$-faces $\overline{S}$ has type $(Y, X)$.

Each cubical Pachner move of type $(X, Y)$ replaces a set $S$ of $n$-faces of an $(n + 1)$-cube by the complementary set of faces $\overline{S}$. It is proved in \cite[Lemma 1]{BEE} that the union of the faces from $S$ is an $n$-ball if and only if $X + Y < n + 1$. Thus, if a type $(X, Y)$ satisfies this condition, the corresponding cubical Pachner move does not change the topology of the $n$-manifold. It only changes the cubulation of this manifold.

\begin{example}
    There are exactly two cubical Pachner moves up to inversion in dimension $n = 1$ (see fig. \ref{Fig:1Moves}). One of them has type $(0, 0)$, and the other has type $(1, 0)$ (its inverse being of type $(0, 1)$).

    \begin{figure}[ht]
        \begin{center}
            \ \hfill
            \begin{tikzpicture}[scale=2.0]
                \begin{scope}
                    \path[main_line] (0.41, 24.823) -- (1.843, 24.823);
                    \path[point] (0.423, 24.823)arc(0.0:90.0:0.025 and -0.025)arc(90.0:180.0:0.025 and -0.025)arc(180.0:270.0:0.025 and -0.025)arc(270.0:360.0:0.025 and -0.025) -- cycle;
                    \path[point] (1.881, 24.823)arc(0.0:90.0:0.025 and -0.025)arc(90.0:180.0:0.025 and -0.025)arc(180.0:270.0:0.025 and -0.025)arc(270.0:360.0:0.025 and -0.025) -- cycle;
                \end{scope}
                \begin{scope}[shift={(0, -0.6)}]
                    \path[main_line] (0.41, 24.385) -- (1.843, 24.385);
                    \path[point] (0.423, 24.385)arc(0.0:90.0:0.025 and -0.025)arc(90.0:180.0:0.025 and -0.025)arc(180.0:270.0:0.025 and -0.025)arc(270.0:360.0:0.025 and -0.025) -- cycle;
                    \path[point] (1.881, 24.385)arc(0.0:90.0:0.025 and -0.025)arc(90.0:180.0:0.025 and -0.025)arc(180.0:270.0:0.025 and -0.025)arc(270.0:360.0:0.025 and -0.025) -- cycle;
                    \path[point] (1.41, 24.385)arc(0.0:90.0:0.025 and -0.025)arc(90.0:180.0:0.025 and -0.025)arc(180.0:270.0:0.025 and -0.025)arc(270.0:360.0:0.025 and -0.025) -- cycle;
                    \path[point] (0.922, 24.385)arc(0.0:90.0:0.025 and -0.025)arc(90.0:180.0:0.025 and -0.025)arc(180.0:270.0:0.025 and -0.025)arc(270.0:360.0:0.025 and -0.025) -- cycle;
                \end{scope}

                \draw[->, thick] (0.75, 24.7) to[bend right, looseness=1.0] node[midway, left] {$(1, 0)$} (0.75, 23.9);

                \draw[<-, thick] (1.5, 24.7) to[bend left, looseness=1.0] node[midway, right] {$(0, 1)$} (1.5, 23.9);
            \end{tikzpicture}
            \hfill
            \begin{tikzpicture}[scale=2.0]
                \begin{scope}
                    \path[main_line] (2.333, 24.848) -- (3.766, 24.848);
                    \path[point] (2.346, 24.848)arc(0.0:90.0:0.025 and -0.025)arc(90.0:180.0:0.025 and -0.025)arc(180.0:270.0:0.025 and -0.025)arc(270.0:360.0:0.025 and -0.025) -- cycle;
                    \path[point] (3.1, 24.848)arc(0.0:90.0:0.025 and -0.025)arc(90.0:180.0:0.025 and -0.025)arc(180.0:270.0:0.025 and -0.025)arc(270.0:360.0:0.025 and -0.025) -- cycle;
                    \path[point] (3.804, 24.848)arc(0.0:90.0:0.025 and -0.025)arc(90.0:180.0:0.025 and -0.025)arc(180.0:270.0:0.025 and -0.025)arc(270.0:360.0:0.025 and -0.025) -- cycle;
                \end{scope}
                \begin{scope}[shift={(0, -1)}]
                    \path[main_line] (2.333, 24.848) -- (3.766, 24.848);
                    \path[point] (2.346, 24.848)arc(0.0:90.0:0.025 and -0.025)arc(90.0:180.0:0.025 and -0.025)arc(180.0:270.0:0.025 and -0.025)arc(270.0:360.0:0.025 and -0.025) -- cycle;
                    \path[point] (3.1, 24.848)arc(0.0:90.0:0.025 and -0.025)arc(90.0:180.0:0.025 and -0.025)arc(180.0:270.0:0.025 and -0.025)arc(270.0:360.0:0.025 and -0.025) -- cycle;
                    \path[point] (3.804, 24.848)arc(0.0:90.0:0.025 and -0.025)arc(90.0:180.0:0.025 and -0.025)arc(180.0:270.0:0.025 and -0.025)arc(270.0:360.0:0.025 and -0.025) -- cycle;
                \end{scope}

                \draw[<->, thick] (3.1, 24.7) to[bend right, looseness=0.0] node[midway, left] {$(0, 0)$} (3.1, 24.0);
            \end{tikzpicture}
            \hfill\ \ 
        \end{center}
        \caption{\label{Fig:1Moves}Cubical Pachner moves for dimension $n = 1$}
    \end{figure}
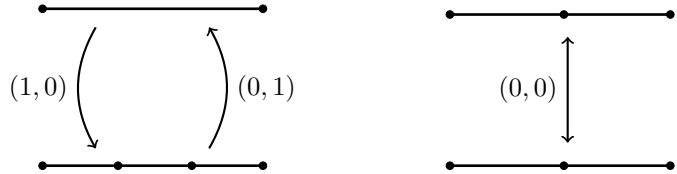

    There are exactly four cubical Pachner moves  up to inversion in dimension $n = 2$ (see fig. \ref{Fig:2Moves}): $(0, 0)$, $(1, 0)$ (opposite to $(0, 1)$), $(1, 1)$, and $(2, 0)$ (opposite to $(0, 2)$).

    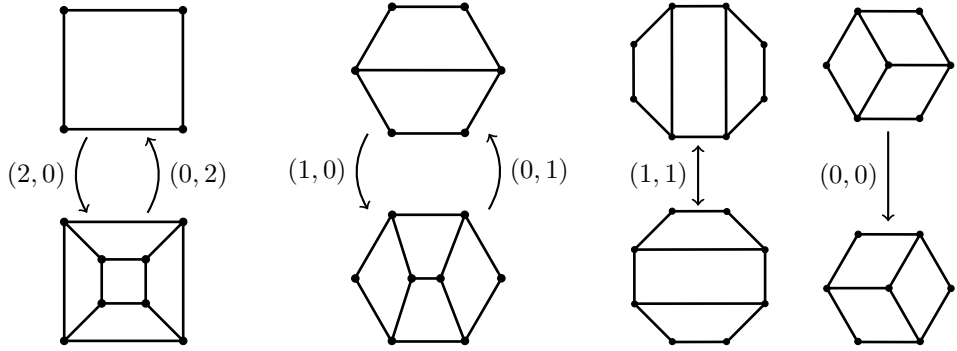
\begin{figure}[ht]
        \begin{center}
            \begin{tikzpicture}[scale=2.0]
                \begin{scope}
                    \path[point] (4.272, 25.24)arc(0.0:90.0:0.025 and -0.025)arc(90.0:180.0:0.025 and -0.025)arc(180.0:270.0:0.025 and -0.025)arc(270.0:360.0:0.025 and -0.025) -- cycle;
                    \path[point] (4.272, 24.453)arc(0.0:90.0:0.025 and -0.025)arc(90.0:180.0:0.025 and -0.025)arc(180.0:270.0:0.025 and -0.025)arc(270.0:360.0:0.025 and -0.025) -- cycle;
                    \path[point] (5.059, 24.453)arc(0.0:90.0:0.025 and -0.025)arc(90.0:180.0:0.025 and -0.025)arc(180.0:270.0:0.025 and -0.025)arc(270.0:360.0:0.025 and -0.025) -- cycle;
                    \path[point] (5.059, 25.24)arc(0.0:90.0:0.025 and -0.025)arc(90.0:180.0:0.025 and -0.025)arc(180.0:270.0:0.025 and -0.025)arc(270.0:360.0:0.025 and -0.025) -- cycle;

                    \path[main_line] (4.247, 25.24) -- (5.034, 25.24) -- (5.034, 24.453) -- (4.247, 24.453) -- cycle;
                \end{scope}
                \begin{scope}[shift={(-1.05, -1.4)}]
                    \path[point] (5.324, 25.24)arc(0.0:90.0:0.025 and -0.025)arc(90.0:180.0:0.025 and -0.025)arc(180.0:270.0:0.025 and -0.025)arc(270.0:360.0:0.025 and -0.025) -- cycle;
                    \path[point] (5.572, 24.992)arc(0.0:90.0:0.025 and -0.025)arc(90.0:180.0:0.025 and -0.025)arc(180.0:270.0:0.025 and -0.025)arc(270.0:360.0:0.025 and -0.025) -- cycle;
                    \path[point] (5.863, 24.992)arc(0.0:90.0:0.025 and -0.025)arc(90.0:180.0:0.025 and -0.025)arc(180.0:270.0:0.025 and -0.025)arc(270.0:360.0:0.025 and -0.025) -- cycle;
                    \path[point] (5.863, 24.701)arc(0.0:90.0:0.025 and -0.025)arc(90.0:180.0:0.025 and -0.025)arc(180.0:270.0:0.025 and -0.025)arc(270.0:360.0:0.025 and -0.025) -- cycle;
                    \path[point] (5.572, 24.701)arc(0.0:90.0:0.025 and -0.025)arc(90.0:180.0:0.025 and -0.025)arc(180.0:270.0:0.025 and -0.025)arc(270.0:360.0:0.025 and -0.025) -- cycle;
                    \path[point] (5.324, 24.453)arc(0.0:90.0:0.025 and -0.025)arc(90.0:180.0:0.025 and -0.025)arc(180.0:270.0:0.025 and -0.025)arc(270.0:360.0:0.025 and -0.025) -- cycle;
                    \path[point] (6.111, 24.453)arc(0.0:90.0:0.025 and -0.025)arc(90.0:180.0:0.025 and -0.025)arc(180.0:270.0:0.025 and -0.025)arc(270.0:360.0:0.025 and -0.025) -- cycle;
                    \path[point] (6.111, 25.24)arc(0.0:90.0:0.025 and -0.025)arc(90.0:180.0:0.025 and -0.025)arc(180.0:270.0:0.025 and -0.025)arc(270.0:360.0:0.025 and -0.025) -- cycle;
                    \path[main_line] (5.299, 25.24) -- (6.086, 25.24) -- (6.086, 24.453) -- (5.299, 24.453) -- cycle;
                    \path[main_line] (5.547, 24.992) -- (5.838, 24.992) -- (5.838, 24.701) -- (5.547, 24.701) -- cycle;
                    \path[main_line] (5.547, 24.992) -- (5.299, 25.24);
                    \path[main_line] (5.838, 24.992) -- (6.086, 25.24);
                    \path[main_line] (5.838, 24.701) -- (6.086, 24.453);
                    \path[main_line] (5.547, 24.701) -- (5.299, 24.453);
                \end{scope}
                \draw[->, thick] (4.4, 24.4) to[bend right, looseness=1.0] node[midway, left] {$(2, 0)$} (4.4, 23.9);

                \draw[<-, thick] (4.8, 24.4) to[bend left, looseness=1.0] node[midway, right] {$(0, 2)$} (4.8, 23.9);
            \end{tikzpicture}
            \hfill
            \begin{tikzpicture}[scale=2.0]
                \begin{scope}
                    \path[point] (7.085, 25.24)arc(0.0:90.0:0.025 and -0.025)arc(90.0:180.0:0.025 and -0.025)arc(180.0:270.0:0.025 and -0.025)arc(270.0:360.0:0.025 and -0.025) -- cycle;
                    \path[point] (7.565, 25.24)arc(0.0:90.0:0.025 and -0.025)arc(90.0:180.0:0.025 and -0.025)arc(180.0:270.0:0.025 and -0.025)arc(270.0:360.0:0.025 and -0.025) -- cycle;
                    \path[point] (7.807, 24.819)arc(0.0:90.0:0.025 and -0.025)arc(90.0:180.0:0.025 and -0.025)arc(180.0:270.0:0.025 and -0.025)arc(270.0:360.0:0.025 and -0.025) -- cycle;
                    \path[point] (7.567, 24.406)arc(0.0:90.0:0.025 and -0.025)arc(90.0:180.0:0.025 and -0.025)arc(180.0:270.0:0.025 and -0.025)arc(270.0:360.0:0.025 and -0.025) -- cycle;
                    \path[point] (7.083, 24.406)arc(0.0:90.0:0.025 and -0.025)arc(90.0:180.0:0.025 and -0.025)arc(180.0:270.0:0.025 and -0.025)arc(270.0:360.0:0.025 and -0.025) -- cycle;
                    \path[point] (6.842, 24.818)arc(0.0:90.0:0.025 and -0.025)arc(90.0:180.0:0.025 and -0.025)arc(180.0:270.0:0.025 and -0.025)arc(270.0:360.0:0.025 and -0.025) -- cycle;
                    
                    \path[main_line,shift={(-0.04, -0.445)}] (7.58, 24.847) -- (7.096, 24.847) -- (6.855, 25.267) -- (7.097, 25.685) -- (7.581, 25.685) -- (7.823, 25.265) -- cycle;
                    \path[main_line] (6.815, 24.821) -- (7.783, 24.82);
                \end{scope}
                \begin{scope}[shift={(-1.29, -1.4)}]
                    \path[point] (8.376, 25.264)arc(0.0:90.0:0.025 and -0.025)arc(90.0:180.0:0.025 and -0.025)arc(180.0:270.0:0.025 and -0.025)arc(270.0:360.0:0.025 and -0.025) -- cycle;
                    \path[point] (8.856, 25.264)arc(0.0:90.0:0.025 and -0.025)arc(90.0:180.0:0.025 and -0.025)arc(180.0:270.0:0.025 and -0.025)arc(270.0:360.0:0.025 and -0.025) -- cycle;
                    \path[point] (9.098, 24.844)arc(0.0:90.0:0.025 and -0.025)arc(90.0:180.0:0.025 and -0.025)arc(180.0:270.0:0.025 and -0.025)arc(270.0:360.0:0.025 and -0.025) -- cycle;
                    \path[point] (8.857, 24.43)arc(0.0:90.0:0.025 and -0.025)arc(90.0:180.0:0.025 and -0.025)arc(180.0:270.0:0.025 and -0.025)arc(270.0:360.0:0.025 and -0.025) -- cycle;
                    \path[point] (8.374, 24.43)arc(0.0:90.0:0.025 and -0.025)arc(90.0:180.0:0.025 and -0.025)arc(180.0:270.0:0.025 and -0.025)arc(270.0:360.0:0.025 and -0.025) -- cycle;
                    \path[point] (8.132, 24.843)arc(0.0:90.0:0.025 and -0.025)arc(90.0:180.0:0.025 and -0.025)arc(180.0:270.0:0.025 and -0.025)arc(270.0:360.0:0.025 and -0.025) -- cycle;
                    \path[point] (8.505, 24.843)arc(0.0:90.0:0.025 and -0.025)arc(90.0:180.0:0.025 and -0.025)arc(180.0:270.0:0.025 and -0.025)arc(270.0:360.0:0.025 and -0.025) -- cycle;
                    \path[point] (8.695, 24.843)arc(0.0:90.0:0.025 and -0.025)arc(90.0:180.0:0.025 and -0.025)arc(180.0:270.0:0.025 and -0.025)arc(270.0:360.0:0.025 and -0.025) -- cycle;
                    
                    \path[main_line] (8.831, 24.426) -- (8.347, 24.426) -- (8.106, 24.846) -- (8.348, 25.264) -- (8.832, 25.264) -- (9.074, 24.844) -- cycle;
                    \path[main_line] (8.48, 24.843) -- (8.67, 24.843);
                    \path[main_line] (8.348, 25.264) -- (8.48, 24.843) -- (8.347, 24.426);
                    \path[main_line] (8.832, 25.264) -- (8.67, 24.843) -- (8.831, 24.426);
                \end{scope}
                \draw[->, thick] (6.9, 24.4) to[bend right, looseness=1.0] node[midway, left] {$(1, 0)$} (6.9, 23.9);

                \draw[<-, thick] (7.7, 24.4) to[bend left, looseness=1.0] node[midway, right] {$(0, 1)$} (7.7, 23.9);
            \end{tikzpicture}
            \hfill
            \begin{tikzpicture}[scale=1.6]
                \begin{scope}
                    \path[point] (4.411, 24.2)arc(0.0:90.0:0.025 and -0.025)arc(90.0:180.0:0.025 and -0.025)arc(180.0:270.0:0.025 and -0.025)arc(270.0:360.0:0.025 and -0.025) -- cycle;
                    \path[point] (4.856, 24.2)arc(0.0:90.0:0.025 and -0.025)arc(90.0:180.0:0.025 and -0.025)arc(180.0:270.0:0.025 and -0.025)arc(270.0:360.0:0.025 and -0.025) -- cycle;
                    \path[point] (4.094, 23.884)arc(0.0:90.0:0.025 and -0.025)arc(90.0:180.0:0.025 and -0.025)arc(180.0:270.0:0.025 and -0.025)arc(270.0:360.0:0.025 and -0.025) -- cycle;
                    \path[point] (4.092, 23.432)arc(0.0:90.0:0.025 and -0.025)arc(90.0:180.0:0.025 and -0.025)arc(180.0:270.0:0.025 and -0.025)arc(270.0:360.0:0.025 and -0.025) -- cycle;
                    \path[point] (4.409, 23.119)arc(0.0:90.0:0.025 and -0.025)arc(90.0:180.0:0.025 and -0.025)arc(180.0:270.0:0.025 and -0.025)arc(270.0:360.0:0.025 and -0.025) -- cycle;
                    \path[point] (4.858, 23.121)arc(0.0:90.0:0.025 and -0.025)arc(90.0:180.0:0.025 and -0.025)arc(180.0:270.0:0.025 and -0.025)arc(270.0:360.0:0.025 and -0.025) -- cycle;
                    \path[point] (5.171, 23.433)arc(0.0:90.0:0.025 and -0.025)arc(90.0:180.0:0.025 and -0.025)arc(180.0:270.0:0.025 and -0.025)arc(270.0:360.0:0.025 and -0.025) -- cycle;
                    \path[point] (5.17, 23.881)arc(0.0:90.0:0.025 and -0.025)arc(90.0:180.0:0.025 and -0.025)arc(180.0:270.0:0.025 and -0.025)arc(270.0:360.0:0.025 and -0.025) -- cycle;
                    
                    \path[main_line,shift={(0.285, -0.4)}] (4.546, 23.519) -- (4.099, 23.519) -- (3.782, 23.835) -- (3.781, 24.283) -- (4.097, 24.599) -- (4.545, 24.6) -- (4.862, 24.284) -- (4.862, 23.836) -- cycle;
                    \path[main_line] (4.383, 24.2) -- (4.384, 23.119);
                    \path[main_line] (4.83, 24.2) -- (4.832, 23.119);
                \end{scope}
                \begin{scope}[shift={(-1.27, -1.7)}]
                    \path[point] (6.424, 23.86)arc(90.0:180.0:0.025 and -0.025)arc(180.0:270.0:0.025 and -0.025)arc(270.0:360.0:0.025 and -0.025)arc(0.0:90.0:0.025 and -0.025) -- cycle;
                    \path[point] (6.424, 23.415)arc(90.0:180.0:0.025 and -0.025)arc(180.0:270.0:0.025 and -0.025)arc(270.0:360.0:0.025 and -0.025)arc(0.0:90.0:0.025 and -0.025) -- cycle;
                    \path[point] (6.108, 24.177)arc(90.0:180.0:0.025 and -0.025)arc(180.0:270.0:0.025 and -0.025)arc(270.0:360.0:0.025 and -0.025)arc(0.0:90.0:0.025 and -0.025) -- cycle;
                    \path[point] (5.656, 24.18)arc(90.0:180.0:0.025 and -0.025)arc(180.0:270.0:0.025 and -0.025)arc(270.0:360.0:0.025 and -0.025)arc(0.0:90.0:0.025 and -0.025) -- cycle;
                    \path[point] (5.344, 23.862)arc(90.0:180.0:0.025 and -0.025)arc(180.0:270.0:0.025 and -0.025)arc(270.0:360.0:0.025 and -0.025)arc(0.0:90.0:0.025 and -0.025) -- cycle;
                    \path[point] (5.345, 23.413)arc(90.0:180.0:0.025 and -0.025)arc(180.0:270.0:0.025 and -0.025)arc(270.0:360.0:0.025 and -0.025)arc(0.0:90.0:0.025 and -0.025) -- cycle;
                    \path[point] (5.657, 23.101)arc(90.0:180.0:0.025 and -0.025)arc(180.0:270.0:0.025 and -0.025)arc(270.0:360.0:0.025 and -0.025)arc(0.0:90.0:0.025 and -0.025) -- cycle;
                    \path[point] (6.105, 23.101)arc(90.0:180.0:0.025 and -0.025)arc(180.0:270.0:0.025 and -0.025)arc(270.0:360.0:0.025 and -0.025)arc(0.0:90.0:0.025 and -0.025) -- cycle;
                    
                    \path[main_line] (5.343, 23.44) -- (5.343, 23.887) -- (5.659, 24.204) -- (6.107, 24.205) -- (6.424, 23.889) -- (6.424, 23.441) -- (6.108, 23.124) -- (5.66, 23.123) -- cycle;
                    \path[main_line] (6.424, 23.889) -- (5.343, 23.887);
                    \path[main_line] (6.424, 23.441) -- (5.343, 23.44);
                \end{scope}

                \draw[<->, thick] (4.6, 23.05) to[bend right, looseness=0.0] node[midway, left] {$(1, 1)$} (4.6, 22.55);
            \end{tikzpicture}
            \hfill
            \begin{tikzpicture}[scale=1.7]
                \begin{scope}
                    \path[point] (7.139, 24.035)arc(0.0:90.0:0.025 and -0.025)arc(90.0:180.0:0.025 and -0.025)arc(180.0:270.0:0.025 and -0.025)arc(270.0:360.0:0.025 and -0.025) -- cycle;
                    \path[point] (7.618, 24.035)arc(0.0:90.0:0.025 and -0.025)arc(90.0:180.0:0.025 and -0.025)arc(180.0:270.0:0.025 and -0.025)arc(270.0:360.0:0.025 and -0.025) -- cycle;
                    \path[point] (7.861, 23.614)arc(0.0:90.0:0.025 and -0.025)arc(90.0:180.0:0.025 and -0.025)arc(180.0:270.0:0.025 and -0.025)arc(270.0:360.0:0.025 and -0.025) -- cycle;
                    \path[point] (7.62, 23.201)arc(0.0:90.0:0.025 and -0.025)arc(90.0:180.0:0.025 and -0.025)arc(180.0:270.0:0.025 and -0.025)arc(270.0:360.0:0.025 and -0.025) -- cycle;
                    \path[point] (7.136, 23.201)arc(0.0:90.0:0.025 and -0.025)arc(90.0:180.0:0.025 and -0.025)arc(180.0:270.0:0.025 and -0.025)arc(270.0:360.0:0.025 and -0.025) -- cycle;
                    \path[point] (7.378, 23.616)arc(0.0:90.0:0.025 and -0.025)arc(90.0:180.0:0.025 and -0.025)arc(180.0:270.0:0.025 and -0.025)arc(270.0:360.0:0.025 and -0.025) -- cycle;
                    \path[point] (6.895, 23.613)arc(0.0:90.0:0.025 and -0.025)arc(90.0:180.0:0.025 and -0.025)arc(180.0:270.0:0.025 and -0.025)arc(270.0:360.0:0.025 and -0.025) -- cycle;
                    
                    \path[main_line] (7.594, 23.196) -- (7.11, 23.197) -- (6.869, 23.616) -- (7.111, 24.035) -- (7.595, 24.034) -- (7.836, 23.615) -- cycle;
                    \path[main_line] (7.111, 24.035) -- (7.352, 23.616) -- (7.11, 23.197);
                    \path[main_line] (7.836, 23.615) -- (7.352, 23.616);
                \end{scope}
                \begin{scope}[shift={(-1.2, -1.65)}]
                    \path[point] (8.342, 23.949)arc(0.0:90.0:0.025 and -0.025)arc(90.0:180.0:0.025 and -0.025)arc(180.0:270.0:0.025 and -0.025)arc(270.0:360.0:0.025 and -0.025) -- cycle;
                    \path[point] (8.822, 23.949)arc(0.0:90.0:0.025 and -0.025)arc(90.0:180.0:0.025 and -0.025)arc(180.0:270.0:0.025 and -0.025)arc(270.0:360.0:0.025 and -0.025) -- cycle;
                    \path[point] (9.064, 23.528)arc(0.0:90.0:0.025 and -0.025)arc(90.0:180.0:0.025 and -0.025)arc(180.0:270.0:0.025 and -0.025)arc(270.0:360.0:0.025 and -0.025) -- cycle;
                    \path[point] (8.824, 23.115)arc(0.0:90.0:0.025 and -0.025)arc(90.0:180.0:0.025 and -0.025)arc(180.0:270.0:0.025 and -0.025)arc(270.0:360.0:0.025 and -0.025) -- cycle;
                    \path[point] (8.34, 23.115)arc(0.0:90.0:0.025 and -0.025)arc(90.0:180.0:0.025 and -0.025)arc(180.0:270.0:0.025 and -0.025)arc(270.0:360.0:0.025 and -0.025) -- cycle;
                    \path[point] (8.581, 23.53)arc(0.0:90.0:0.025 and -0.025)arc(90.0:180.0:0.025 and -0.025)arc(180.0:270.0:0.025 and -0.025)arc(270.0:360.0:0.025 and -0.025) -- cycle;
                    \path[point] (8.098, 23.527)arc(0.0:90.0:0.025 and -0.025)arc(90.0:180.0:0.025 and -0.025)arc(180.0:270.0:0.025 and -0.025)arc(270.0:360.0:0.025 and -0.025) -- cycle;
                    
                    \path[main_line] (8.797, 23.11) -- (8.313, 23.111) -- (8.072, 23.53) -- (8.315, 23.949) -- (8.798, 23.948) -- (9.04, 23.529) -- cycle;
                    \path[main_line] (8.798, 23.948) -- (8.556, 23.53) -- (8.797, 23.11);
                    \path[main_line] (8.072, 23.53) -- (8.556, 23.53);
                \end{scope}

                \draw[->, thick] (7.35, 23.1) to[bend right, looseness=0.0] node[midway, left] {$(0, 0)$} (7.35, 22.4);
            \end{tikzpicture}
        \end{center}
        \caption{\label{Fig:2Moves}Cubical Pachner moves for dimension $n = 2$}
    \end{figure}

    There are exactly six cubical Pachner moves up to inversion in dimension $n = 3$: $(0, 0)$, $(1, 0)$ (opposite to $(0, 1)$), $(1, 1)$, $(2, 0)$ (opposite to $(0, 2)$), $(2, 1)$ (opposite to $(1, 2)$), and $(3, 0)$ (opposite to $(0, 3)$).
\end{example}

% \begin{remark}
    % There is some ambiguity in the designation of the cubical Pachner move. If it has the type $(X, X)$, then it replace one set of faces of the type $(X, X)$ to another set of faces with the same type $(X, X)$. And in this designation is not clear what the first combination of faces and what is the second one. So, in each particular case it requires more detailed description of interchangeable combinations.
% \end{remark}

In addition to local cubical Pachner moves, we will consider one non-local transformation of a cubulation --- \emph{subdivision}. It replaces every $n$-cube of the cubulation with $2^n$ cubes by splitting each interval $I = [-1, 1]$ into two intervals $[-1, 0]$ and $[0, 1]$, and taking the direct product of all these parts (see fig. \ref{Fig:Subdivision}). The inverse transformation is called \emph{de-subdivision}.

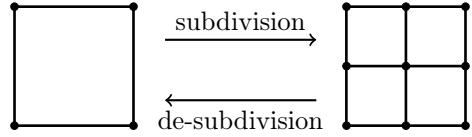
\begin{figure}[ht]
    \begin{center}
        \begin{tikzpicture}[scale=2.0]
            \begin{scope}
                \path[point] (0.529, 23.419)arc(0.0:90.0:0.025 and -0.025)arc(90.0:180.0:0.025 and -0.025)arc(180.0:270.0:0.025 and -0.025)arc(270.0:360.0:0.025 and -0.025) -- cycle;
                \path[point] (0.529, 22.632)arc(0.0:90.0:0.025 and -0.025)arc(90.0:180.0:0.025 and -0.025)arc(180.0:270.0:0.025 and -0.025)arc(270.0:360.0:0.025 and -0.025) -- cycle;
                \path[point] (1.316, 22.632)arc(0.0:90.0:0.025 and -0.025)arc(90.0:180.0:0.025 and -0.025)arc(180.0:270.0:0.025 and -0.025)arc(270.0:360.0:0.025 and -0.025) -- cycle;
                \path[point] (1.316, 23.419)arc(0.0:90.0:0.025 and -0.025)arc(90.0:180.0:0.025 and -0.025)arc(180.0:270.0:0.025 and -0.025)arc(270.0:360.0:0.025 and -0.025) -- cycle;
                
                \path[main_line] (0.504, 23.419) -- (1.291, 23.419) -- (1.291, 22.632) -- (0.504, 22.632) -- cycle;
            \end{scope}
            \begin{scope}[shift={(1.0, 0.0)}]
                \path[point] (1.725, 23.419)arc(0.0:90.0:0.025 and -0.025)arc(90.0:180.0:0.025 and -0.025)arc(180.0:270.0:0.025 and -0.025)arc(270.0:360.0:0.025 and -0.025) -- cycle;
                \path[point] (2.118, 23.419)arc(0.0:90.0:0.025 and -0.025)arc(90.0:180.0:0.025 and -0.025)arc(180.0:270.0:0.025 and -0.025)arc(270.0:360.0:0.025 and -0.025) -- cycle;
                \path[point] (1.725, 23.026)arc(0.0:90.0:0.025 and -0.025)arc(90.0:180.0:0.025 and -0.025)arc(180.0:270.0:0.025 and -0.025)arc(270.0:360.0:0.025 and -0.025) -- cycle;
                \path[point] (2.118, 22.632)arc(0.0:90.0:0.025 and -0.025)arc(90.0:180.0:0.025 and -0.025)arc(180.0:270.0:0.025 and -0.025)arc(270.0:360.0:0.025 and -0.025) -- cycle;
                \path[point] (2.512, 23.026)arc(0.0:90.0:0.025 and -0.025)arc(90.0:180.0:0.025 and -0.025)arc(180.0:270.0:0.025 and -0.025)arc(270.0:360.0:0.025 and -0.025) -- cycle;
                \path[point] (2.118, 23.026)arc(0.0:90.0:0.025 and -0.025)arc(90.0:180.0:0.025 and -0.025)arc(180.0:270.0:0.025 and -0.025)arc(270.0:360.0:0.025 and -0.025) -- cycle;
                \path[point] (1.725, 22.632)arc(0.0:90.0:0.025 and -0.025)arc(90.0:180.0:0.025 and -0.025)arc(180.0:270.0:0.025 and -0.025)arc(270.0:360.0:0.025 and -0.025) -- cycle;
                \path[point] (2.512, 22.632)arc(0.0:90.0:0.025 and -0.025)arc(90.0:180.0:0.025 and -0.025)arc(180.0:270.0:0.025 and -0.025)arc(270.0:360.0:0.025 and -0.025) -- cycle;
                \path[point] (2.512, 23.419)arc(0.0:90.0:0.025 and -0.025)arc(90.0:180.0:0.025 and -0.025)arc(180.0:270.0:0.025 and -0.025)arc(270.0:360.0:0.025 and -0.025) -- cycle;
                
                \path[main_line] (1.699, 23.419) -- (2.487, 23.419) -- (2.487, 22.632) -- (1.699, 22.632) -- cycle;
                \path[main_line] (2.093, 23.419) -- (2.093, 23.026) -- (2.093, 22.632);
                \path[main_line] (1.699, 23.026) -- (2.093, 23.026) -- (2.487, 23.026);
            \end{scope}
            \draw[->, thick] (1.5, 23.2) -- (2.5, 23.2) node[midway, above] {subdivision};
            \draw[<-, thick] (1.5, 22.8) -- (2.5, 22.8) node[midway, below] {de-subdivision};
        \end{tikzpicture}
    \end{center}
    \caption{\label{Fig:Subdivision}Subdivision/de-subdivision transformations}
\end{figure}

\begin{example}
    Subdivision replaces each 2-cube with four 2-cubes and each 3-cube with eight 3-cubes.
\end{example}

\section{Quadrangulation transformations}

\begin{lemma}
    \label{Lemma:MatchTransforms}
    Let $F$ be a surface (2-manifold), let $\mathcal{T}_1$ and $\mathcal{T}_2$ be two triangulations of $F$, and let $\mathcal{C}_1 = \beta(\mathcal{T}_1)$ and $\mathcal{C}_2 = \beta(\mathcal{T}_2)$ be quadrangulations of $F$ that match the triangulations $\mathcal{T}_1$ and $\mathcal{T}_2$ respectively. Then there exists a finite sequence of cubical Pachner moves of types $(0, 0)$, $(0, 1)$, $(1, 0)$, $(0, 2)$ and $(2, 0)$ that transforms $\mathcal{C}_1$ into $\mathcal{C}_2$.
\end{lemma}
\begin{remark}
    The move of type $(1, 1)$ is not required to transform any quadrangulation of the surface that matches some triangulation into any other quadrangulation that also matches some triangulation.
\end{remark}
\begin{proof} 
    It is well known that any two triangulations of the same surface can be transformed into one another by a finite sequence of three classical Pachner moves: the 2--2-move, the 1--3-move, and the 3--1-move, shown in figure \ref{Fig:SimplexMoves} (see \cite{FH,L,P,RST}). For triangulations, the moves are denoted as $x$--$y$, where $x$ is the number of triangles that are replaced by $y$ other triangles. Thus, it is sufficient to prove the lemma in the cases where $\mathcal{T}_2$ is obtained from $\mathcal{T}_1$ by either a 2--2-move or a 3--1-move.
    
    \begin{figure}[ht]
        \begin{center}
            \ \hfill
            \begin{tikzpicture}[scale=1.0]
                \begin{scope}
                    \path[point] (1.222, 22.249)arc(0.0:90.0:0.025 and -0.025)arc(90.0:180.0:0.025 and -0.025)arc(180.0:270.0:0.025 and -0.025)arc(270.0:360.0:0.025 and -0.025) -- cycle;
                    \path[point] (0.368, 21.393)arc(0.0:90.0:0.025 and -0.025)arc(90.0:180.0:0.025 and -0.025)arc(180.0:270.0:0.025 and -0.025)arc(270.0:360.0:0.025 and -0.025) -- cycle;
                    \path[point] (1.228, 20.541)arc(0.0:90.0:0.025 and -0.025)arc(90.0:180.0:0.025 and -0.025)arc(180.0:270.0:0.025 and -0.025)arc(270.0:360.0:0.025 and -0.025) -- cycle;
                    \path[point] (2.081, 21.393)arc(0.0:90.0:0.025 and -0.025)arc(90.0:180.0:0.025 and -0.025)arc(180.0:270.0:0.025 and -0.025)arc(270.0:360.0:0.025 and -0.025) -- cycle;
                    
                    \path[main_line] (1.197, 22.255) -- (2.061, 21.39) -- (1.202, 20.531) -- (0.338, 21.396) -- cycle;
                    \path[main_line] (1.197, 22.255) -- (1.202, 20.531);
                \end{scope}
                \begin{scope}[shift={(0.0, -1.0)}]
                    \path[point] (1.222, 20.195)arc(0.0:90.0:0.025 and -0.025)arc(90.0:180.0:0.025 and -0.025)arc(180.0:270.0:0.025 and -0.025)arc(270.0:360.0:0.025 and -0.025) -- cycle;
                    \path[point] (0.37, 19.342)arc(359.999:90.0:0.025 and -0.025)arc(90.0:180.001:0.025 and -0.025)arc(180.001:270.0:0.025 and -0.025)arc(270.0:359.999:0.025 and -0.025) -- cycle;
                    \path[point] (2.075, 19.337)arc(0.001:90.0:0.025 and -0.025)arc(90.0:179.999:0.025 and -0.025)arc(179.999:270.0:0.025 and -0.025)arc(270.0:0.001:0.025 and -0.025) -- cycle;
                    \path[point] (1.228, 18.484)arc(359.999:90.0:0.025 and -0.025)arc(90.0:180.001:0.025 and -0.025)arc(180.001:270.0:0.025 and -0.025)arc(270.0:359.999:0.025 and -0.025) -- cycle;
                    
                    \path[main_line] (1.197, 20.201) -- (2.061, 19.337) -- (1.202, 18.478) -- (0.338, 19.343) -- cycle;
                    \path[main_line] (0.338, 19.343) -- (2.061, 19.337);
                \end{scope}

                \draw[<->, thick] (1.2, 20.3) to[bend right, looseness=0.0] node[midway, left] {2--2} (1.2, 19.4);
            \end{tikzpicture}
            \hfill
            \begin{tikzpicture}
                \begin{scope}
                    \path[point] (2.624, 20.795)arc(0.0:90.0:0.025 and -0.025)arc(90.0:180.0:0.025 and -0.025)arc(180.0:270.0:0.025 and -0.025)arc(270.0:360.0:0.025 and -0.025) -- cycle;
                    \path[point] (4.265, 20.793)arc(0.0:90.0:0.025 and -0.025)arc(90.0:180.0:0.025 and -0.025)arc(180.0:270.0:0.025 and -0.025)arc(270.0:360.0:0.025 and -0.025) -- cycle;
                    \path[point] (3.446, 22.21)arc(0.0:90.0:0.025 and -0.025)arc(90.0:180.0:0.025 and -0.025)arc(180.0:270.0:0.025 and -0.025)arc(270.0:360.0:0.025 and -0.025) -- cycle;
                    
                    \path[main_line,cm={ 1.114,-0.0,-0.0,1.114,(-0.29, -3.319)}] (2.589, 21.643) -- (3.33, 22.928) -- (4.072, 21.643) -- cycle;
                \end{scope}
                \begin{scope}[shift={(0.0, -1.0)}]
                    \path[point] (3.446, 19.186)arc(0.001:90.0:0.025 and -0.025)arc(90.0:179.999:0.025 and -0.025)arc(179.999:270.0:0.025 and -0.025)arc(270.0:0.001:0.025 and -0.025) -- cycle;
                    \path[point] (2.623, 18.654)arc(359.999:90.0:0.025 and -0.025)arc(90.0:180.001:0.025 and -0.025)arc(180.001:270.0:0.025 and -0.025)arc(270.0:359.999:0.025 and -0.025) -- cycle;
                    \path[point] (3.447, 20.074)arc(0.0:90.0:0.025 and -0.025)arc(90.0:180.0:0.025 and -0.025)arc(180.0:270.0:0.025 and -0.025)arc(270.0:360.0:0.025 and -0.025) -- cycle;
                    \path[point] (4.265, 18.657)arc(0.001:90.0:0.025 and -0.025)arc(90.0:179.999:0.025 and -0.025)arc(179.999:270.0:0.025 and -0.025)arc(270.0:0.001:0.025 and -0.025) -- cycle;
                    
                    \path[main_line] (2.595, 18.651) -- (3.421, 20.082) -- (4.247, 18.651) -- cycle;
                    \path[main_line] (2.595, 18.651) -- (3.421, 19.186);
                    \path[main_line] (3.421, 20.082) -- (3.421, 19.186);
                    \path[main_line] (4.235, 18.657) -- (3.421, 19.186);
                \end{scope}

                \draw[->, thick] (2.9, 20.5) to[bend right, looseness=1.0] node[midway, left] {1--3} (2.9, 19.0);

                \draw[<-, thick] (3.9, 20.5) to[bend left, looseness=1.0] node[midway, right] {3--1} (3.9, 19.0);
            \end{tikzpicture}
            \hfill \ \ 
        \end{center}
        \caption{\label{Fig:SimplexMoves}Classical Pachner moves for triangulations of surfaces}
    \end{figure}
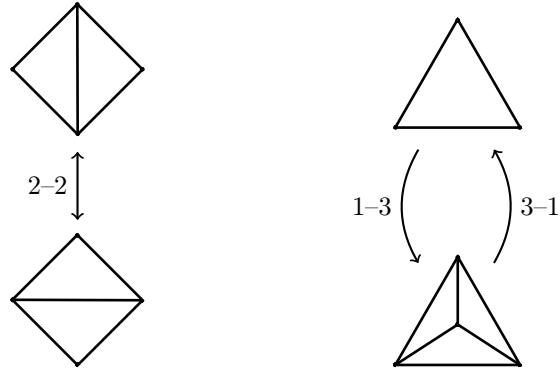

    \emph{2--2-move.} In this case, quadrangulation $\mathcal{C}_1$ can be transformed into quadrangulation $\mathcal{C}_2$ by a sequence of four moves: we start with a move of type $(0, 0)$, then apply a move of type $(0, 1)$, next apply the inverse of a move of type $(1, 0)$ in another region, and finally apply a move of type $(0, 0)$ (see fig. \ref{Fig:MatchedMoves22}, where the regions in which each transformation is applied are shaded gray).

    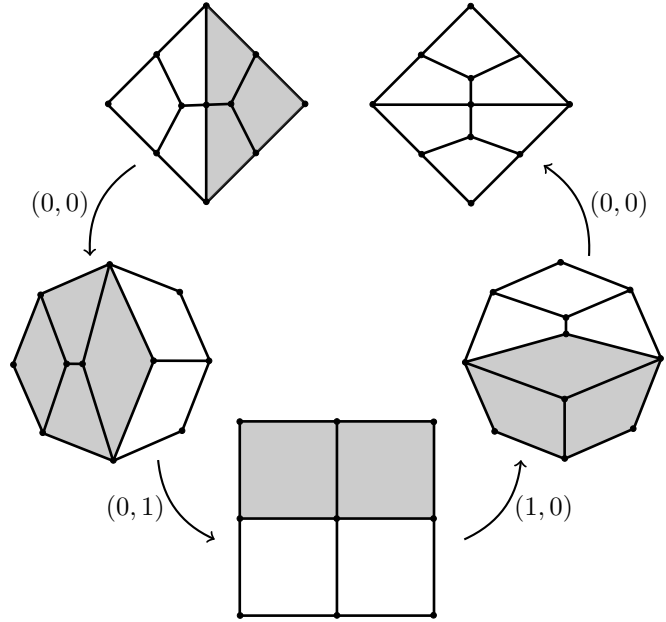
\begin{figure}[ht]
        \begin{center}
            \begin{tikzpicture}[scale=1.5]
                % \def\R{2}
                % \coordinate (A) at ({cos(126)*\R}, {sin(126)*\R});
                % \coordinate (B) at ({cos(54)*\R},  {sin(54)*\R});
                % \coordinate (C) at ({cos(342)*\R}, {sin(342)*\R});
                % \coordinate (D) at ({cos(270)*\R}, {sin(270)*\R});
                % \coordinate (E) at ({cos(198)*\R}, {sin(198)*\R});

                % \draw[thick] (A) -- (B) -- (C) -- (D) -- (E) -- cycle;

                \begin{scope}[shift={(-6.55, -20.3)}]
                    \path[point] (5.4, 22.807)arc(0.0:90.0:0.025 and -0.025)arc(90.0:180.0:0.025 and -0.025)arc(180.0:270.0:0.025 and -0.025)arc(270.0:360.0:0.025 and -0.025) -- cycle;
                    \path[point] (4.962, 22.377)arc(0.0:90.0:0.025 and -0.025)arc(90.0:180.0:0.025 and -0.025)arc(180.0:270.0:0.025 and -0.025)arc(270.0:360.0:0.025 and -0.025) -- cycle;
                    \path[point] (4.535, 21.941)arc(0.0:90.0:0.025 and -0.025)arc(90.0:180.0:0.025 and -0.025)arc(180.0:270.0:0.025 and -0.025)arc(270.0:360.0:0.025 and -0.025) -- cycle;
                    \path[point] (5.181, 21.925)arc(0.0:90.0:0.025 and -0.025)arc(90.0:180.0:0.025 and -0.025)arc(180.0:270.0:0.025 and -0.025)arc(270.0:360.0:0.025 and -0.025) -- cycle;

                    \path[main_line] (5.374, 22.816) -- (6.249, 21.941) -- (5.374, 21.066) -- (4.5, 21.941) -- cycle;
                    \path[fill=gray, fill opacity=0.4] (5.374, 22.816) -- (5.374, 21.066) -- (6.249, 21.941) -- (5.374, 22.816);
                    \path[main_line] (5.374, 22.816) -- (5.374, 21.066);
                    \path[main_line] (5.591, 21.941) -- (5.16, 21.925) -- (4.937, 22.378);
                    \path[main_line] (5.16, 21.925) -- (4.937, 21.504);
                    \path[main_line] (5.812, 22.378) -- (5.591, 21.941) -- (5.812, 21.504);
                    \path[point] (5.4, 21.932)arc(0.0:90.0:0.025 and -0.025)arc(90.0:180.0:0.025 and -0.025)arc(180.0:270.0:0.025 and -0.025)arc(270.0:360.0:0.025 and -0.025) -- cycle;
                    \path[point] (4.961, 21.507)arc(0.0:90.0:0.025 and -0.025)arc(90.0:180.0:0.025 and -0.025)arc(180.0:270.0:0.025 and -0.025)arc(270.0:360.0:0.025 and -0.025) -- cycle;
                    \path[point] (5.4, 21.078)arc(0.0:90.0:0.025 and -0.025)arc(90.0:180.0:0.025 and -0.025)arc(180.0:270.0:0.025 and -0.025)arc(270.0:360.0:0.025 and -0.025) -- cycle;
                    \path[point] (5.837, 21.508)arc(0.0:90.0:0.025 and -0.025)arc(90.0:180.0:0.025 and -0.025)arc(180.0:270.0:0.025 and -0.025)arc(270.0:360.0:0.025 and -0.025) -- cycle;
                    \path[point] (5.618, 21.941)arc(0.0:90.0:0.025 and -0.025)arc(90.0:180.0:0.025 and -0.025)arc(180.0:270.0:0.025 and -0.025)arc(270.0:360.0:0.025 and -0.025) -- cycle;
                    \path[point] (6.271, 21.941)arc(0.0:90.0:0.025 and -0.025)arc(90.0:180.0:0.025 and -0.025)arc(180.0:270.0:0.025 and -0.025)arc(270.0:360.0:0.025 and -0.025) -- cycle;
                    \path[point] (5.837, 22.377)arc(0.0:90.0:0.025 and -0.025)arc(90.0:180.0:0.025 and -0.025)arc(180.0:270.0:0.025 and -0.025)arc(270.0:360.0:0.025 and -0.025) -- cycle;
                \end{scope}

                \begin{scope}[shift={(-9.3, -22.6)}]
                    \path[fill=gray, fill opacity=0.4] (6.663, 22.562) -- (6.42, 21.943) -- (6.685, 21.335) -- (7.304, 21.092) -- (7.662, 21.976) -- (7.272, 22.827) -- cycle;
                    \path[main_line,cm={ 0.938,0.27,-0.27,0.938,(7.588, -0.709)}] (6.573, 21.634) -- (5.898, 21.545) -- (5.359, 21.96) -- (5.27, 22.634) -- (5.684, 23.174) -- (6.358, 23.263) -- (6.898, 22.849) -- (6.987, 22.174) -- cycle;
                    \path[main_line] (6.663, 22.562) -- (6.893, 21.949) -- (6.685, 21.335);
                    \path[main_line] (7.272, 22.827) -- (7.035, 21.948) -- (7.304, 21.092);
                    \path[main_line] (8.155, 21.975) -- (7.662, 21.976);
                    \path[main_line] (7.272, 22.827) -- (7.662, 21.976) -- (7.304, 21.092);
                    \path[main_line] (6.893, 21.949) -- (7.035, 21.948);
                    \path[point] (6.449, 21.941)arc(0.0:90.0:0.025 and -0.025)arc(90.0:180.0:0.025 and -0.025)arc(180.0:270.0:0.025 and -0.025)arc(270.0:360.0:0.025 and -0.025) -- cycle;
                    \path[point] (6.689, 22.557)arc(0.0:90.0:0.025 and -0.025)arc(90.0:180.0:0.025 and -0.025)arc(180.0:270.0:0.025 and -0.025)arc(270.0:360.0:0.025 and -0.025) -- cycle;
                    \path[point] (7.298, 22.825)arc(0.0:90.0:0.025 and -0.025)arc(90.0:180.0:0.025 and -0.025)arc(180.0:270.0:0.025 and -0.025)arc(270.0:360.0:0.025 and -0.025) -- cycle;
                    \path[point] (6.707, 21.341)arc(0.0:90.0:0.025 and -0.025)arc(90.0:180.0:0.025 and -0.025)arc(180.0:270.0:0.025 and -0.025)arc(270.0:360.0:0.025 and -0.025) -- cycle;
                    \path[point] (7.33, 21.097)arc(0.0:90.0:0.025 and -0.025)arc(90.0:180.0:0.025 and -0.025)arc(180.0:270.0:0.025 and -0.025)arc(270.0:360.0:0.025 and -0.025) -- cycle;
                    \path[point] (7.936, 21.359)arc(0.0:90.0:0.025 and -0.025)arc(90.0:180.0:0.025 and -0.025)arc(180.0:270.0:0.025 and -0.025)arc(270.0:360.0:0.025 and -0.025) -- cycle;
                    \path[point] (8.174, 21.975)arc(0.0:90.0:0.025 and -0.025)arc(90.0:180.0:0.025 and -0.025)arc(180.0:270.0:0.025 and -0.025)arc(270.0:360.0:0.025 and -0.025) -- cycle;
                    \path[point] (7.916, 22.581)arc(0.0:90.0:0.025 and -0.025)arc(90.0:180.0:0.025 and -0.025)arc(180.0:270.0:0.025 and -0.025)arc(270.0:360.0:0.025 and -0.025) -- cycle;
                    \path[point] (6.918, 21.948)arc(0.0:90.0:0.025 and -0.025)arc(90.0:180.0:0.025 and -0.025)arc(180.0:270.0:0.025 and -0.025)arc(270.0:360.0:0.025 and -0.025) -- cycle;
                    \path[point] (7.058, 21.948)arc(0.0:90.0:0.025 and -0.025)arc(90.0:180.0:0.025 and -0.025)arc(180.0:270.0:0.025 and -0.025)arc(270.0:360.0:0.025 and -0.025) -- cycle;
                    \path[point] (7.685, 21.976)arc(0.0:90.0:0.025 and -0.025)arc(90.0:180.0:0.025 and -0.025)arc(180.0:270.0:0.025 and -0.025)arc(270.0:360.0:0.025 and -0.025) -- cycle;
                \end{scope}

                \begin{scope}[shift={(-9.35, -24)}]
                    \path[fill=gray, fill opacity=0.4] (8.47, 22.841) -- (8.47, 21.984) -- (10.183, 21.984) -- (10.183, 22.841) -- (8.47, 22.841);
                    \path[main_line] (8.47, 22.841) -- (10.183, 22.841) -- (10.183, 21.128) -- (8.47, 21.128) -- cycle;
                    \path[main_line] (9.327, 22.841) -- (9.327, 21.128);
                    \path[main_line] (8.47, 21.984) -- (10.183, 21.984);
                    \path[point] (8.495, 22.839)arc(0.0:90.0:0.025 and -0.025)arc(90.0:180.0:0.025 and -0.025)arc(180.0:270.0:0.025 and -0.025)arc(270.0:360.0:0.025 and -0.025) -- cycle;
                    \path[point] (8.496, 21.984)arc(0.0:90.0:0.025 and -0.025)arc(90.0:180.0:0.025 and -0.025)arc(180.0:270.0:0.025 and -0.025)arc(270.0:360.0:0.025 and -0.025) -- cycle;
                    \path[point] (8.498, 21.129)arc(0.0:90.0:0.025 and -0.025)arc(90.0:180.0:0.025 and -0.025)arc(180.0:270.0:0.025 and -0.025)arc(270.0:360.0:0.025 and -0.025) -- cycle;
                    \path[point] (9.352, 21.13)arc(0.0:90.0:0.025 and -0.025)arc(90.0:180.0:0.025 and -0.025)arc(180.0:270.0:0.025 and -0.025)arc(270.0:360.0:0.025 and -0.025) -- cycle;
                    \path[point] (10.204, 21.131)arc(0.0:89.999:0.025 and -0.025)arc(89.999:180.0:0.025 and -0.025)arc(180.0:270.001:0.025 and -0.025)arc(270.001:360.0:0.025 and -0.025) -- cycle;
                    \path[point] (10.204, 21.984)arc(0.0:89.999:0.025 and -0.025)arc(89.999:180.0:0.025 and -0.025)arc(180.0:270.001:0.025 and -0.025)arc(270.001:360.0:0.025 and -0.025) -- cycle;
                    \path[point] (9.35, 21.984)arc(0.0:90.0:0.025 and -0.025)arc(90.0:180.0:0.025 and -0.025)arc(180.0:270.0:0.025 and -0.025)arc(270.0:360.0:0.025 and -0.025) -- cycle;
                    \path[point] (9.352, 22.84)arc(0.0:90.0:0.025 and -0.025)arc(90.0:180.0:0.025 and -0.025)arc(180.0:270.0:0.025 and -0.025)arc(270.0:360.0:0.025 and -0.025) -- cycle;
                    \path[point] (10.206, 22.837)arc(0.0:89.999:0.025 and -0.025)arc(89.999:180.0:0.025 and -0.025)arc(180.0:270.001:0.025 and -0.025)arc(270.001:360.0:0.025 and -0.025) -- cycle;
                \end{scope}

                \begin{scope}[shift={(-9.6, -22.6)}]
                    \path[fill=gray, fill opacity=0.4] (10.701, 21.962) -- (10.967, 21.353) -- (11.585, 21.11) -- (12.194, 21.375) -- (12.437, 21.993) -- (11.597, 22.212) -- cycle;
                    \path[main_line,cm={ 0.938,0.27,-0.27,0.938,(11.87, -0.691)}] (6.573, 21.634) -- (5.898, 21.545) -- (5.359, 21.96) -- (5.27, 22.634) -- (5.684, 23.174) -- (6.358, 23.263) -- (6.898, 22.849) -- (6.987, 22.174) -- cycle;
                    \path[main_line] (11.585, 21.11) -- (11.585, 21.64);
                    \path[main_line] (10.701, 21.962) -- (11.585, 21.64) -- (12.437, 21.993);
                    \path[main_line] (10.944, 22.58) -- (11.597, 22.358) -- (12.171, 22.602);
                    \path[main_line] (10.701, 21.962) -- (11.597, 22.212) -- (12.437, 21.993);
                    \path[main_line] (11.597, 22.358) -- (11.597, 22.212);
                    \path[point] (10.979, 22.58)arc(0.0:90.001:0.025 and -0.025)arc(90.001:180.0:0.025 and -0.025)arc(180.0:269.999:0.025 and -0.025)arc(269.999:360.0:0.025 and -0.025) -- cycle;
                    \path[point] (11.575, 22.844)arc(0.0:90.001:0.025 and -0.025)arc(90.001:180.0:0.025 and -0.025)arc(180.0:269.999:0.025 and -0.025)arc(269.999:360.0:0.025 and -0.025) -- cycle;
                    \path[point] (12.188, 22.6)arc(0.0:89.999:0.025 and -0.025)arc(89.999:180.0:0.025 and -0.025)arc(180.0:270.001:0.025 and -0.025)arc(270.001:360.0:0.025 and -0.025) -- cycle;
                    \path[point] (12.456, 21.996)arc(0.0:90.001:0.025 and -0.025)arc(90.001:180.0:0.025 and -0.025)arc(180.0:269.999:0.025 and -0.025)arc(269.999:360.0:0.025 and -0.025) -- cycle;
                    \path[point] (12.215, 21.377)arc(0.0:90.001:0.025 and -0.025)arc(90.001:180.0:0.025 and -0.025)arc(180.0:269.999:0.025 and -0.025)arc(269.999:360.0:0.025 and -0.025) -- cycle;
                    \path[point] (11.611, 21.114)arc(0.0:89.999:0.025 and -0.025)arc(89.999:180.0:0.025 and -0.025)arc(180.0:270.001:0.025 and -0.025)arc(270.001:360.0:0.025 and -0.025) -- cycle;
                    \path[point] (10.993, 21.353)arc(0.0:90.001:0.025 and -0.025)arc(90.001:180.0:0.025 and -0.025)arc(180.0:269.999:0.025 and -0.025)arc(269.999:360.0:0.025 and -0.025) -- cycle;
                    \path[point] (10.734, 21.962)arc(0.0:89.999:0.025 and -0.025)arc(89.999:180.0:0.025 and -0.025)arc(180.0:270.001:0.025 and -0.025)arc(270.001:360.0:0.025 and -0.025) -- cycle;
                    \path[point] (11.622, 22.357)arc(0.0:89.999:0.025 and -0.025)arc(89.999:180.0:0.025 and -0.025)arc(180.0:270.001:0.025 and -0.025)arc(270.001:360.0:0.025 and -0.025) -- cycle;
                    \path[point] (11.623, 22.212)arc(0.0:89.999:0.025 and -0.025)arc(89.999:180.0:0.025 and -0.025)arc(180.0:270.001:0.025 and -0.025)arc(270.001:360.0:0.025 and -0.025) -- cycle;
                    \path[point] (11.61, 21.638)arc(0.0:90.001:0.025 and -0.025)arc(90.001:180.0:0.025 and -0.025)arc(180.0:269.999:0.025 and -0.025)arc(269.999:360.0:0.025 and -0.025) -- cycle;
                \end{scope}

                \begin{scope}[shift={(-12.6, -20.36)}]
                    \path[main_line] (13.758, 22.869) -- (14.633, 21.994) -- (13.758, 21.119) -- (12.884, 21.994) -- cycle;
                    \path[main_line] (12.884, 21.994) -- (14.633, 21.994);
                    \path[main_line] (13.321, 22.431) -- (13.758, 22.225) -- (14.196, 22.431);
                    \path[main_line] (13.321, 21.557) -- (13.758, 21.713).. controls (14.222, 21.547) and (14.041, 21.612) .. (14.196, 21.557);
                    \path[main_line] (13.758, 22.225) -- (13.758, 21.713);
                    \path[point] (13.782, 22.863)arc(0.0:89.999:0.025 and -0.025)arc(89.999:180.0:0.025 and -0.025)arc(180.0:270.001:0.025 and -0.025)arc(270.001:360.0:0.025 and -0.025) -- cycle;
                    \path[point] (13.345, 22.433)arc(0.0:89.999:0.025 and -0.025)arc(89.999:180.0:0.025 and -0.025)arc(180.0:270.001:0.025 and -0.025)arc(270.001:360.0:0.025 and -0.025) -- cycle;
                    \path[point] (12.918, 21.997)arc(0.0:89.999:0.025 and -0.025)arc(89.999:180.0:0.025 and -0.025)arc(180.0:270.001:0.025 and -0.025)arc(270.001:360.0:0.025 and -0.025) -- cycle;
                    \path[point] (13.349, 21.557)arc(0.0:90.001:0.025 and -0.025)arc(90.001:180.0:0.025 and -0.025)arc(180.0:269.999:0.025 and -0.025)arc(269.999:360.0:0.025 and -0.025) -- cycle;
                    \path[point] (13.785, 21.125)arc(0.0:89.999:0.025 and -0.025)arc(89.999:180.0:0.025 and -0.025)arc(180.0:270.001:0.025 and -0.025)arc(270.001:360.0:0.025 and -0.025) -- cycle;
                    \path[point] (14.216, 21.557)arc(0.0:90.001:0.025 and -0.025)arc(90.001:180.0:0.025 and -0.025)arc(180.0:269.999:0.025 and -0.025)arc(269.999:360.0:0.025 and -0.025) -- cycle;
                    \path[point] (13.784, 22.227)arc(0.0:90.001:0.025 and -0.025)arc(90.001:180.0:0.025 and -0.025)arc(180.0:269.999:0.025 and -0.025)arc(269.999:360.0:0.025 and -0.025) -- cycle;
                    \path[point] (13.783, 21.996)arc(0.0:90.001:0.025 and -0.025)arc(90.001:180.0:0.025 and -0.025)arc(180.0:269.999:0.025 and -0.025)arc(269.999:360.0:0.025 and -0.025) -- cycle;
                    \path[point] (13.781, 21.71)arc(0.0:90.001:0.025 and -0.025)arc(90.001:180.0:0.025 and -0.025)arc(180.0:269.999:0.025 and -0.025)arc(269.999:360.0:0.025 and -0.025) -- cycle;
                    \path[point] (14.654, 21.997)arc(0.0:90.001:0.025 and -0.025)arc(90.001:180.0:0.025 and -0.025)arc(180.0:269.999:0.025 and -0.025)arc(269.999:360.0:0.025 and -0.025) -- cycle;
                \end{scope}

                \draw[->, thick] (-1.8, 1.1) to[bend right] node[midway, left] {$(0, 0)$} (-2.2, 0.3);

                \draw[->, thick] (-1.6, -1.5) to[bend right] node[midway, left] {$(0, 1)$} (-1.1, -2.2);

                \draw[->, thick] (1.1, -2.2) to[bend right] node[midway, right] {$(1, 0)$} (1.6, -1.5);
                
                \draw[->, thick] (2.2, 0.3) to[bend right] node[midway, right] {$(0, 0)$} (1.8, 1.1);
            \end{tikzpicture}
        \end{center}
        \caption{\label{Fig:MatchedMoves22}Transformations of the quadrangulation to implement the 2--2-move}
    \end{figure}

    \emph{3--1-move.} In this case, the required sequence is shown in fig. \ref{Fig:MatchMoves31}. It also contains four moves: we start with a sequence of two moves of type $(0, 0)$, then a move of type $(0, 1)$, and finally remove the square by a move of type $(0, 2)$.

    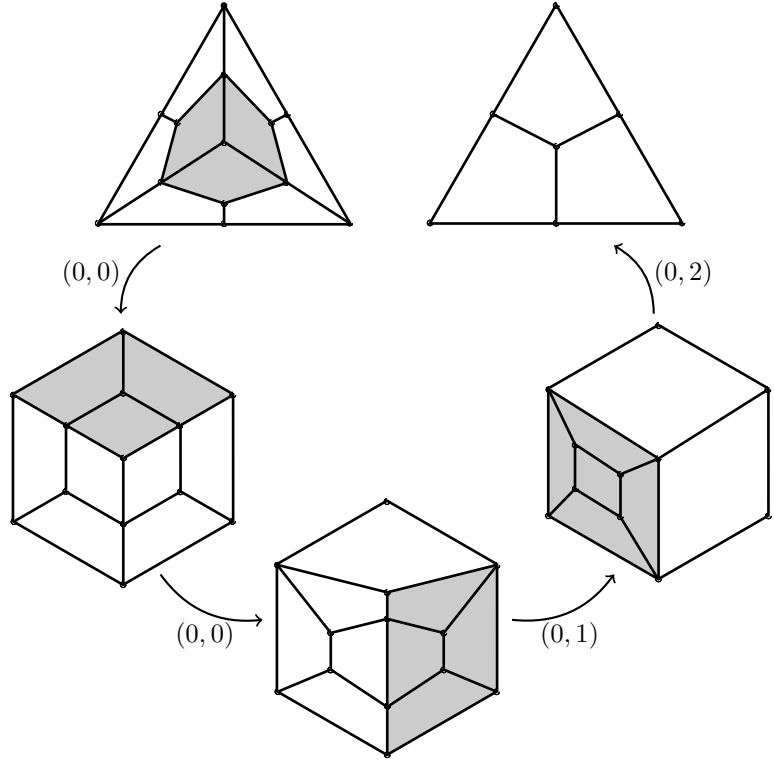
\begin{figure}[ht]
        \begin{center}
            \begin{tikzpicture}[scale=1.5]
                % \def\R{2.5}
                % \coordinate (A) at ({cos(126)*\R}, {sin(126)*\R});
                % \coordinate (B) at ({cos(54)*\R},  {sin(54)*\R});
                % \coordinate (C) at ({cos(342)*\R}, {sin(342)*\R});
                % \coordinate (D) at ({cos(270)*\R}, {sin(270)*\R});
                % \coordinate (E) at ({cos(198)*\R}, {sin(198)*\R});

                % \draw[thick] (A) -- (B) -- (C) -- (D) -- (E) -- cycle;

                \begin{scope}[shift={(-7.07, -17.45)}]
                    \path[fill=gray, fill opacity=0.4] (5.631, 20.069) -- (5.216, 19.634) -- (5.074, 19.104) -- (5.631, 18.922) -- (6.188, 19.104) -- (6.047, 19.634) -- (5.631, 20.069);
                    \path[point] (5.655, 20.667)arc(0.0:90.0:0.025 and -0.025)arc(90.0:180.0:0.025 and -0.025)arc(180.0:270.0:0.025 and -0.025)arc(270.0:360.0:0.025 and -0.025) -- cycle;
                    \path[point] (5.101, 19.708)arc(360.0:90.0:0.025 and -0.025)arc(90.0:180.0:0.025 and -0.025)arc(180.0:270.0:0.025 and -0.025)arc(270.0:360.0:0.025 and -0.025) -- cycle;
                    \path[point] (4.546, 18.747)arc(0.001:90.0:0.025 and -0.025)arc(90.0:179.999:0.025 and -0.025)arc(179.999:270.0:0.025 and -0.025)arc(270.0:0.001:0.025 and -0.025) -- cycle;
                    \path[point] (6.763, 18.749)arc(0.001:90.0:0.025 and -0.025)arc(90.0:179.999:0.025 and -0.025)arc(179.999:270.0:0.025 and -0.025)arc(270.0:0.001:0.025 and -0.025) -- cycle;
                    \path[point] (5.657, 18.744)arc(0.001:90.0:0.025 and -0.025)arc(90.0:179.999:0.025 and -0.025)arc(179.999:270.0:0.025 and -0.025)arc(270.0:0.001:0.025 and -0.025) -- cycle;
                    \path[point] (5.101, 19.11)arc(359.999:90.0:0.025 and -0.025)arc(90.0:180.001:0.025 and -0.025)arc(180.001:270.0:0.025 and -0.025)arc(270.0:359.999:0.025 and -0.025) -- cycle;
                    \path[point] (5.658, 18.921)arc(0.001:90.0:0.025 and -0.025)arc(90.0:179.999:0.025 and -0.025)arc(179.999:270.0:0.025 and -0.025)arc(270.0:0.001:0.025 and -0.025) -- cycle;
                    \path[point] (6.203, 19.107)arc(0.001:90.0:0.025 and -0.025)arc(90.0:179.999:0.025 and -0.025)arc(179.999:270.0:0.025 and -0.025)arc(270.0:0.001:0.025 and -0.025) -- cycle;
                    \path[point] (6.213, 19.707)arc(0.001:90.0:0.025 and -0.025)arc(90.0:179.999:0.025 and -0.025)arc(179.999:270.0:0.025 and -0.025)arc(270.0:0.001:0.025 and -0.025) -- cycle;
                    \path[point] (6.074, 19.633)arc(0.001:90.0:0.025 and -0.025)arc(90.0:179.999:0.025 and -0.025)arc(179.999:270.0:0.025 and -0.025)arc(270.0:0.001:0.025 and -0.025) -- cycle;
                    \path[point] (5.657, 20.057)arc(0.001:90.0:0.025 and -0.025)arc(90.0:179.999:0.025 and -0.025)arc(179.999:270.0:0.025 and -0.025)arc(270.0:0.001:0.025 and -0.025) -- cycle;
                    \path[point] (5.242, 19.633)arc(0.001:90.0:0.025 and -0.025)arc(90.0:179.999:0.025 and -0.025)arc(179.999:270.0:0.025 and -0.025)arc(270.0:0.001:0.025 and -0.025) -- cycle;
                    \path[point] (5.657, 19.458)arc(0.001:90.0:0.025 and -0.025)arc(90.0:179.999:0.025 and -0.025)arc(179.999:270.0:0.025 and -0.025)arc(270.0:0.001:0.025 and -0.025) -- cycle;
                    \path[main_line,cm={ 0.931,-0.0,-0.0,0.931,(0.134, 0.445)}] (4.709, 19.649) -- (5.907, 21.723) -- (7.104, 19.649) -- cycle;
                    \path[main_line] (4.517, 18.742) -- (5.631, 19.466) -- (5.631, 20.672);
                    \path[main_line] (5.631, 19.466) -- (6.746, 18.742);
                    \path[main_line] (5.074, 19.707) -- (5.216, 19.634) -- (5.074, 19.104);
                    \path[main_line] (5.216, 19.634) -- (5.631, 20.069);
                    \path[main_line] (6.188, 19.707) -- (6.047, 19.634) -- (6.188, 19.104);
                    \path[main_line] (6.047, 19.634) -- (5.631, 20.069);
                    \path[main_line] (5.074, 19.104) -- (5.631, 18.922) -- (6.188, 19.104);
                    \path[main_line] (5.631, 18.922) -- (5.631, 18.742);
                \end{scope}

                \begin{scope}[shift={(-10.5, -20.5)}]
                    \path[fill=gray, fill opacity=0.4] (7.201, 20.285) -- (8.172, 19.724) -- (9.144, 20.285) -- (8.172, 20.846) -- cycle;
                    \path[point] (8.196, 18.61)arc(0.001:90.0:0.025 and -0.025)arc(90.0:179.999:0.025 and -0.025)arc(179.999:270.0:0.025 and -0.025)arc(270.0:0.001:0.025 and -0.025) -- cycle;
                    \path[point] (7.233, 19.165)arc(0.001:90.0:0.025 and -0.025)arc(90.0:179.999:0.025 and -0.025)arc(179.999:270.0:0.025 and -0.025)arc(270.0:0.001:0.025 and -0.025) -- cycle;
                    \path[point] (7.697, 20.013)arc(0.001:90.0:0.025 and -0.025)arc(90.0:179.999:0.025 and -0.025)arc(179.999:270.0:0.025 and -0.025)arc(270.0:0.001:0.025 and -0.025) -- cycle;
                    \path[point] (8.196, 20.299)arc(0.001:90.0:0.025 and -0.025)arc(90.0:179.999:0.025 and -0.025)arc(179.999:270.0:0.025 and -0.025)arc(270.0:0.001:0.025 and -0.025) -- cycle;
                    \path[point] (8.695, 20.012)arc(0.001:90.0:0.025 and -0.025)arc(90.0:179.999:0.025 and -0.025)arc(179.999:270.0:0.025 and -0.025)arc(270.0:0.001:0.025 and -0.025) -- cycle;
                    \path[point] (8.199, 19.726)arc(0.001:90.0:0.025 and -0.025)arc(90.0:179.999:0.025 and -0.025)arc(179.999:270.0:0.025 and -0.025)arc(270.0:0.001:0.025 and -0.025) -- cycle;
                    \path[point] (7.691, 19.432)arc(0.001:90.0:0.025 and -0.025)arc(90.0:179.999:0.025 and -0.025)arc(179.999:270.0:0.025 and -0.025)arc(270.0:0.001:0.025 and -0.025) -- cycle;
                    \path[point] (8.197, 19.146)arc(0.001:90.0:0.025 and -0.025)arc(90.0:179.999:0.025 and -0.025)arc(179.999:270.0:0.025 and -0.025)arc(270.0:0.001:0.025 and -0.025) -- cycle;
                    \path[point] (8.7, 19.428)arc(0.001:90.0:0.025 and -0.025)arc(90.0:179.999:0.025 and -0.025)arc(179.999:270.0:0.025 and -0.025)arc(270.0:0.001:0.025 and -0.025) -- cycle;
                    \path[point] (7.229, 20.282)arc(0.001:90.0:0.025 and -0.025)arc(90.0:179.999:0.025 and -0.025)arc(179.999:270.0:0.025 and -0.025)arc(270.0:0.001:0.025 and -0.025) -- cycle;
                    \path[point] (8.197, 20.841)arc(0.001:90.0:0.025 and -0.025)arc(90.0:179.999:0.025 and -0.025)arc(179.999:270.0:0.025 and -0.025)arc(270.0:0.001:0.025 and -0.025) -- cycle;
                    \path[point] (9.163, 20.282)arc(0.001:90.0:0.025 and -0.025)arc(90.0:179.999:0.025 and -0.025)arc(179.999:270.0:0.025 and -0.025)arc(270.0:0.001:0.025 and -0.025) -- cycle;
                    \path[point] (9.161, 19.168)arc(0.001:90.0:0.025 and -0.025)arc(90.0:179.999:0.025 and -0.025)arc(179.999:270.0:0.025 and -0.025)arc(270.0:0.001:0.025 and -0.025) -- cycle;
                    \path[main_line,cm={ 0.819,-0.473,0.473,0.819,(-7.519, 6.743)}] (8.09, 19.16) -- (6.905, 19.16) -- (6.312, 20.187) -- (6.905, 21.213) -- (8.09, 21.213) -- (8.683, 20.187) -- cycle;
                    \path[main_line,cm={ 0.426,-0.246,0.246,0.426,(0.012, 12.966)}] (8.09, 19.16) -- (6.905, 19.16) -- (6.312, 20.187) -- (6.905, 21.213) -- (8.09, 21.213) -- (8.683, 20.187) -- cycle;
                    \path[main_line] (7.201, 20.285) -- (8.172, 19.724);
                    \path[main_line] (8.172, 18.602) -- (8.172, 19.724);
                    \path[main_line] (9.144, 20.285) -- (8.172, 19.724);
                    \path[main_line] (8.172, 20.307) -- (8.172, 20.846);
                    \path[main_line] (7.667, 19.432) -- (7.201, 19.163);
                    \path[main_line] (8.677, 19.432) -- (9.144, 19.163);
                \end{scope}

                \begin{scope}[shift={(-10.62, -22.0)}]
                    \path[fill=gray, fill opacity=0.4] (10.62, 20.042) -- (10.62, 18.602) -- (11.592, 19.163) -- (11.592, 20.285) -- cycle;
                    \path[point] (9.678, 20.282)arc(0.001:90.0:0.025 and -0.025)arc(90.0:179.999:0.025 and -0.025)arc(179.999:270.0:0.025 and -0.025)arc(270.0:0.001:0.025 and -0.025) -- cycle;
                    \path[point] (10.643, 20.843)arc(0.001:89.999:0.025 and -0.025)arc(89.999:179.999:0.025 and -0.025)arc(179.999:270.001:0.025 and -0.025)arc(270.001:0.001:0.025 and -0.025) -- cycle;
                    \path[point] (10.646, 18.61)arc(0.001:90.001:0.025 and -0.025)arc(90.001:179.999:0.025 and -0.025)arc(179.999:269.999:0.025 and -0.025)arc(269.999:0.001:0.025 and -0.025) -- cycle;
                    \path[point] (9.681, 19.166)arc(0.001:90.0:0.025 and -0.025)arc(90.0:179.999:0.025 and -0.025)arc(179.999:270.0:0.025 and -0.025)arc(270.0:0.001:0.025 and -0.025) -- cycle;
                    \path[point] (10.645, 20.04)arc(0.001:90.001:0.025 and -0.025)arc(90.001:179.999:0.025 and -0.025)arc(179.999:269.999:0.025 and -0.025)arc(269.999:0.001:0.025 and -0.025) -- cycle;
                    \path[point] (10.146, 19.685)arc(0.001:90.001:0.025 and -0.025)arc(90.001:179.999:0.025 and -0.025)arc(179.999:269.999:0.025 and -0.025)arc(269.999:0.001:0.025 and -0.025) -- cycle;
                    \path[point] (10.145, 19.357)arc(0.001:90.001:0.025 and -0.025)arc(90.001:179.999:0.025 and -0.025)arc(179.999:269.999:0.025 and -0.025)arc(269.999:0.001:0.025 and -0.025) -- cycle;
                    \path[point] (10.646, 19.037)arc(0.001:90.001:0.025 and -0.025)arc(90.001:179.999:0.025 and -0.025)arc(179.999:269.999:0.025 and -0.025)arc(269.999:0.001:0.025 and -0.025) -- cycle;
                    \path[point] (11.143, 19.358)arc(0.001:90.001:0.025 and -0.025)arc(90.001:179.999:0.025 and -0.025)arc(179.999:269.999:0.025 and -0.025)arc(269.999:0.001:0.025 and -0.025) -- cycle;
                    \path[point] (11.144, 19.683)arc(0.001:89.999:0.025 and -0.025)arc(89.999:179.999:0.025 and -0.025)arc(179.998:270.001:0.025 and -0.025)arc(270.001:0.002:0.025 and -0.025) -- cycle;
                    \path[point] (10.644, 19.807)arc(0.001:89.999:0.025 and -0.025)arc(89.999:179.999:0.025 and -0.025)arc(179.999:270.001:0.025 and -0.025)arc(270.001:0.001:0.025 and -0.025) -- cycle;
                    \path[point] (11.612, 20.279)arc(0.001:90.001:0.025 and -0.025)arc(90.001:179.999:0.025 and -0.025)arc(179.999:269.999:0.025 and -0.025)arc(269.999:0.001:0.025 and -0.025) -- cycle;
                    \path[point] (11.609, 19.165)arc(0.001:90.001:0.025 and -0.025)arc(90.001:179.999:0.025 and -0.025)arc(179.999:269.999:0.025 and -0.025)arc(269.999:0.001:0.025 and -0.025) -- cycle;
                    \path[main_line,cm={ 0.819,-0.473,0.473,0.819,(-5.071, 6.743)}] (8.09, 19.16) -- (6.905, 19.16) -- (6.312, 20.187) -- (6.905, 21.213) -- (8.09, 21.213) -- (8.683, 20.187) -- cycle;
                    \path[main_line] (9.648, 20.285) -- (10.62, 20.042) -- (11.592, 20.285);
                    \path[main_line] (10.62, 20.042) -- (10.62, 18.602);
                    \path[main_line] (10.62, 19.808) -- (10.122, 19.684) -- (10.122, 19.361) -- (10.62, 19.031);
                    \path[main_line] (10.122, 19.684) -- (9.648, 20.285);
                    \path[main_line] (10.122, 19.361) -- (9.648, 19.163);
                    \path[main_line] (10.62, 19.808) -- (11.118, 19.684) -- (11.118, 19.361) -- (10.62, 19.031);
                    \path[main_line] (11.118, 19.684) -- (11.592, 20.285);
                    \path[main_line] (11.118, 19.361) -- (11.592, 19.163);
                \end{scope}

                \begin{scope}[shift={(-10.65, -20.45)}]
                    \path[fill=gray, fill opacity=0.4] (12.071, 20.285) -- (12.071, 19.163) -- (13.043, 18.602) -- (13.043, 19.672) -- (12.071, 20.285);
                    \path[point] (12.101, 20.28)arc(0.001:89.999:0.025 and -0.025)arc(89.999:179.999:0.025 and -0.025)arc(179.999:270.001:0.025 and -0.025)arc(270.001:0.001:0.025 and -0.025) -- cycle;
                    \path[point] (12.098, 19.167)arc(0.001:89.999:0.025 and -0.025)arc(89.999:179.999:0.025 and -0.025)arc(179.999:270.001:0.025 and -0.025)arc(270.001:0.001:0.025 and -0.025) -- cycle;
                    \path[point] (13.067, 18.61)arc(0.001:89.999:0.025 and -0.025)arc(89.999:179.999:0.025 and -0.025)arc(179.999:270.001:0.025 and -0.025)arc(270.001:0.001:0.025 and -0.025) -- cycle;
                    \path[point] (14.037, 19.166)arc(0.001:89.999:0.025 and -0.025)arc(89.999:179.999:0.025 and -0.025)arc(179.999:270.001:0.025 and -0.025)arc(270.001:0.001:0.025 and -0.025) -- cycle;
                    \path[point] (14.032, 20.28)arc(0.001:90.001:0.025 and -0.025)arc(90.001:179.999:0.025 and -0.025)arc(179.999:269.999:0.025 and -0.025)arc(269.999:0.001:0.025 and -0.025) -- cycle;
                    \path[point] (13.069, 20.845)arc(0.001:89.999:0.025 and -0.025)arc(89.999:179.999:0.025 and -0.025)arc(179.999:270.001:0.025 and -0.025)arc(270.001:0.001:0.025 and -0.025) -- cycle;
                    \path[point] (13.063, 19.669)arc(0.001:90.001:0.025 and -0.025)arc(90.001:179.999:0.025 and -0.025)arc(179.999:269.999:0.025 and -0.025)arc(269.999:0.001:0.025 and -0.025) -- cycle;
                    \path[point] (12.732, 19.532)arc(0.001:89.999:0.025 and -0.025)arc(89.999:179.999:0.025 and -0.025)arc(179.999:270.001:0.025 and -0.025)arc(270.001:0.001:0.025 and -0.025) -- cycle;
                    \path[point] (12.333, 19.791)arc(0.001:90.001:0.025 and -0.025)arc(90.001:179.999:0.025 and -0.025)arc(179.999:269.999:0.025 and -0.025)arc(269.999:0.001:0.025 and -0.025) -- cycle;
                    \path[point] (12.334, 19.406)arc(0.001:89.999:0.025 and -0.025)arc(89.999:179.999:0.025 and -0.025)arc(179.999:270.001:0.025 and -0.025)arc(270.001:0.001:0.025 and -0.025) -- cycle;
                    \path[point] (12.73, 19.159)arc(0.001:89.999:0.025 and -0.025)arc(89.999:179.999:0.025 and -0.025)arc(179.999:270.001:0.025 and -0.025)arc(270.001:0.001:0.025 and -0.025) -- cycle;
                    \path[main_line,cm={ 0.819,-0.473,0.473,0.819,(-2.648, 6.743)}] (8.09, 19.16) -- (6.905, 19.16) -- (6.312, 20.187) -- (6.905, 21.213) -- (8.09, 21.213) -- (8.683, 20.187) -- cycle;
                    \path[main_line] (12.071, 20.285) -- (13.043, 19.672) -- (13.043, 18.602);
                    \path[main_line] (13.043, 19.672) -- (14.014, 20.285);
                    \path[main_line] (12.306, 19.798) -- (12.306, 19.404) -- (12.707, 19.156) -- (12.707, 19.533) -- cycle;
                    \path[main_line] (12.306, 19.798) -- (12.071, 20.285);
                    \path[main_line] (12.306, 19.404) -- (12.071, 19.163);
                    \path[main_line] (12.707, 19.156) -- (13.043, 18.602);
                    \path[main_line] (12.707, 19.533) -- (13.043, 19.672);
                \end{scope}

                \begin{scope}[shift={(-14.22, -17.4)}]
                    \path[point] (15.735, 20.617)arc(0.001:90.001:0.025 and -0.025)arc(90.001:179.999:0.025 and -0.025)arc(179.999:269.999:0.025 and -0.025)arc(269.999:0.001:0.025 and -0.025) -- cycle;
                    \path[point] (14.623, 18.7)arc(0.001:90.001:0.025 and -0.025)arc(90.001:179.999:0.025 and -0.025)arc(179.999:269.999:0.025 and -0.025)arc(269.999:0.001:0.025 and -0.025) -- cycle;
                    \path[point] (16.844, 18.7)arc(0.001:90.001:0.025 and -0.025)arc(90.001:179.999:0.025 and -0.025)arc(179.999:269.999:0.025 and -0.025)arc(269.999:0.001:0.025 and -0.025) -- cycle;
                    \path[point] (15.179, 19.661)arc(0.001:90.001:0.025 and -0.025)arc(90.001:179.999:0.025 and -0.025)arc(179.999:269.999:0.025 and -0.025)arc(269.999:0.001:0.025 and -0.025) -- cycle;
                    \path[point] (16.293, 19.659)arc(0.001:90.001:0.025 and -0.025)arc(90.001:179.999:0.025 and -0.025)arc(179.999:269.999:0.025 and -0.025)arc(269.999:0.001:0.025 and -0.025) -- cycle;
                    \path[point] (15.737, 19.375)arc(0.001:90.001:0.025 and -0.025)arc(90.001:179.999:0.025 and -0.025)arc(179.999:269.999:0.025 and -0.025)arc(269.999:0.001:0.025 and -0.025) -- cycle;
                    \path[point] (15.735, 18.698)arc(0.001:90.001:0.025 and -0.025)arc(90.001:179.999:0.025 and -0.025)arc(179.999:269.999:0.025 and -0.025)arc(269.999:0.001:0.025 and -0.025) -- cycle;
                    \path[main_line] (14.595, 18.695) -- (15.71, 20.625) -- (16.824, 18.695) -- cycle;
                    \path[main_line] (15.152, 19.66) -- (15.71, 19.373) -- (16.267, 19.66);
                    \path[main_line] (15.71, 19.373) -- (15.71, 18.695);
                \end{scope}

                \draw[->, thick] (-2.0, 1.1) to[bend right] node[midway, left] {$(0, 0)$} (-2.35, 0.5);

                \draw[->, thick] (-2.0, -1.8) to[bend right] node[midway, below] {$(0, 0)$} (-1.1, -2.2);

                \draw[->, thick] (1.1, -2.2) to[bend right] node[midway, below] {$(0, 1)$} (2.0, -1.8);
                
                \draw[->, thick] (2.35, 0.5) to[bend right] node[midway, right] {$(0, 2)$} (2.0, 1.1);
            \end{tikzpicture}
        \end{center}
        \caption{\label{Fig:MatchMoves31}Transformations of the quadrangulation to implement the 3--1-move}
    \end{figure}
\end{proof}

\begin{lemma}
    \label{Lemma:Subdivide}
    Let $\mathcal{C}$ be a quadrangulation of the surface $F$, and let $\mathcal{C}'$ be the result of the subdivision of $\mathcal{C}$. Then there exists a triangulation $\mathcal{T}$ of $F$ such that $\mathcal{C}'$ can be transformed into $\beta(\mathcal{T})$ by a finite sequence of cubical Pachner moves of types $(0, 0)$, $(0, 1)$, $(1, 0)$, $(0, 2)$ and $(2, 0)$.
\end{lemma}
\begin{proof}
    Each group of four squares resulting from the subdivision of an original square of $\mathcal{C}$ can be transformed into six squares that form a quadrangulation matching some triangulation: a move of type $(1, 0)$ and then a move of type $(0, 0)$ (in fig. \ref{Fig:MatchedMoves22}, this uses half of the transformations, from the central quadrangulation to either the initial or the final quadrangulation).
\end{proof}

\begin{theorem}
    \label{Theorem:Transforms}
    Let $\mathcal{C}_1$ and $\mathcal{C}_2$ be two quadrangulations of the surface $F$. Then there is a finite sequence of subdivisions and de-subdivisions and cubical Pachner moves of types $(0, 0)$, $(0, 1)$, $(1, 0)$, $(0, 2)$ and $(2, 0)$ that transform $\mathcal{C}_1$ into $\mathcal{C}_2$.
\end{theorem}
\begin{proof}
    The theorem follows from Lemmas \ref{Lemma:MatchTransforms} and \ref{Lemma:Subdivide}. Indeed, we start from the quadrangulation $\mathcal{C}_1$, apply one subdivision, and then use the required cubical Pachner moves to obtain a quadrangulation $\mathcal{C}'$ that matches some triangulation $\mathcal{T}_1$ (by Lemma \ref{Lemma:Subdivide}). We do the same with $\mathcal{C}_2$ and obtain a quadrangulation $\mathcal{C}_2'$ that matches a triangulation $\mathcal{T}_2$. Then, by Lemma \ref{Lemma:MatchTransforms}, we can transform $\mathcal{C}'$ into $\mathcal{C}_2'$ using cubical Pachner moves, after which we apply the reverse transformations from $\mathcal{C}_2'$ back to $\mathcal{C}_2$ (the last step being a de-subdivision).
\end{proof}

\begin{remark}
    To transform $\mathcal{C}_1$ into $\mathcal{C}_2$, it is sufficient to apply only one subdivision and one de-subdivision.
\end{remark}


\begin{thebibliography}{9}
    \bibitem{A1} G. Amendola, A 3-manifold complexity via immersed surfaces // J. Knot Theory Ramifications, 2010, V. 19, No. 1549.

    \bibitem{A2} G. Amendola, Orientable closed 3-manifolds with surface-complexity one // arXiv:1011.4196.

    \bibitem{BEE} Bern M., Eppstein D., Erickson J., Flipping Cubical Meshes // Engineering with Computers, 2001, V. 18, No. 3, P. 173–187.

    \bibitem{FH} Friedl S., Hannes J., Pachner’s Theorem // 2021-2022 MATRIX Annals. MATRIX Book Series, 2024, V. 5, Springer.

    \bibitem{F1} Funar L., Cubulations mod bubble moves // Contemporary Mathematics. Proceedings of a Conference on Low Dimensional Topology (January 12-17, 1998, Funchal, Madeira, Portugal), 1999, V. 233, P. 29–43.

    \bibitem{F2} Funar L., Surface cubications mod flips // Manuscripta Mathematica, 2008, V. 125, P. 285–307.

    \bibitem{KK} Korablev Ph.G., Kazakov A.A., Manifolds of cubic complexity two // Siberian Electronic Mathematical Reports, 2016, V. 13, P. 1–15.
    
    \bibitem{L} Lickorish W.B.R., Simplicial moves on complexes and manifolds // Geometry \& Topology Monographs Volume 2: Proceedings of the Kirbyfest, 1999, P. 299–320.

    \bibitem{N} Nakamoto A., Diagonal Transformations in Quadrangulations of Surfaces // Journal of Graph Theory, 1996, V. 21, No. 3, P. 289-299.

    \bibitem{P} Pachner U., P.L. Homeomorphic Manifolds are Equivalent by Elementary Shellings // European Journal of Combinatorics, 1991, V. 12, P. 129-145.

    \bibitem{RST} Rubinstein J.H., Segerman H., Tillmann S., Traversing three-manifold triangulations and spines // arXiv:1812.02806.

    


\end{thebibliography}
\end{document}